\documentclass{mfat}
\pagespan{95}{116}

\usepackage{mmap}
\usepackage[utf8]{inputenc}
\newtheorem{thm}{Theorem}[section]
\newtheorem{prop}[thm]{Proposition}
\newtheorem{lem}[thm]{Lemma}
\newtheorem{cor}[thm]{Corollary}
\newtheorem{rem}[thm]{Remark}
\newtheorem{defi}[thm]{Definition}
\newtheorem{exa}[thm]{Example}

\newcommand{\sH}{\mathcal{H}}
\newcommand{\sK}{\mathcal{K}}

\newcommand{\sF}{\mathcal{F}}
\newcommand{\sU}{\mathcal{U}}

\newcommand{\sG}{\mathcal{G}}
\newcommand{\sL}{\mathcal{L}}

\newcommand{\sD}{\mathcal{D}}

\newcommand{\sX}{\mathcal{X}}

\newcommand{\sY}{\mathcal{Y}}

\newcommand{\sS}{\mathcal{S}}

\def\a{\alpha}
\def\b{\beta}
\def\d{\delta}
\def\de{\Delta}

\def\ga{\Gamma}

\def\l{\lambda}

\def\si{\Sigma}
\def\t{\tau}
\def\va{\varphi}

\def\th{\theta}

\def\z{\zeta}
\def\ts{\times}

\def\iy{\infty}
\def\im{{\rm Im\, }}

\def\kr{{\rm Ker\, }}

\def\linhull{{\rm span}}

\def\lg{\langle}
\def\rg{\rangle}
\def\wh{\widehat}
\def\wt{\widetilde}

\def\BC{{\mathbb C}}
\def\BD{{\mathbb D}}

\def\BT{{\mathbb T}}

\renewcommand{\theequation}{\arabic{section}.\arabic{equation}}

\newcommand{\ands}{\quad\mbox{and}\quad}

\newcommand{\epr}{{\hfill $\Box$}}
\newcommand{\nn}{{\nonumber}}
\newcommand{\bu}{{\bullet}}

\begin{document}

\title[Generalized solutions of Riccati equalities and inequalities]
    {Generalized solutions of Riccati equalities and inequalities}

\author{D.~Z.~Arov}
\address{Division of Applied Mathematics and Informatics, Institute of Physics and Mathematics,
South-Ukrainian Pedagogical University,  Odessa,  65020, Ukraine}
\email{arov\_damir@mail.ru}

\author{M.~A.~Kaashoek}
\address{Department of Mathematics, Vrije Universiteit, Amsterdam, The Netherlands}
\email{m.a.kaashoek@vu.nl}

\author{D.~R.~Pik}
\address{Faculty of Social and Behavioural Sciences, University of Amsterdam, Amsterdam, The Netherlands}
\email{drpik2@uva.nl}

\subjclass[2000]{Primary 47A48, 47A62; Secondary 47A56, 93C55}
\date{07/02/2016; \ \  Revised 03/03/2016}

\dedicatory{Dedicated to Yurii Makarovich Berezanskii on the
occasion of his 90th birthday}

\keywords{Discrete time-invariant systems, scattering supply rate,
passive systems, Riccati equality, Riccati inequality,
Kalman--Yakubovich--Popov inequality}

\begin{abstract}The Riccati inequality and equality are studied for infinite dimensional linear discrete
time stationary systems with respect to the scattering supply
rate. The results  obtained  are an addition to and  based on our
earlier work on the  Kalman--Yakubovich--Popov inequality in
\cite{AKP06}. The main theorems are closely related  to the
results of Yu.~M.~Arlinski\u{\i} in \cite{Arl08}. The main
difference is that we do not assume the original system to be a
passive scattering system, and we allow the solutions of  the
Riccati inequality and equality  to  satisfy weaker  conditions.
\end{abstract}

\maketitle

\setcounter{equation}{0}
\section{Introduction and main theorems}\label{sec:thms}

This paper is an addition to  \cite{AKP06}. Throughout
$\si = (A, B, C, D; \sX, \sU, \sY)$  is  a shorthand notation for the
linear discrete time-invariant  system
\begin{equation}
\label{syst}
\Sigma
\left\{
\begin{array}{rcl}
x_{n+1} &=& A x_n + B u_n
\\
\noalign{\vskip4pt}
y_n &=& C x_n + D u_n
\end{array}
\right.
\quad  (n=0, 1,2 , \ldots).
\end{equation}
Here
$A:\sX \rightarrow \sX$, $B:\sU \rightarrow \sX$, $C:\sX \rightarrow \sY$ and $D: \sU \rightarrow \sY$ are bounded linear operators acting between separable Hilbert spaces. The operator $A$ is called   the \emph{state operator},   $B$ and $C$   are referred to as     \emph{input operator} and \emph{output operator}, respectively, and   $D$  is called the   \emph{feed through operator}. The spaces $\sX$, $\sU$, and $\sY$ are called \emph{state space}, \emph{input space}, and \emph{output space}, respectively. By definition  the \emph{transfer function} of the system $\si$ is the operator-valued function
 \[
 \th_\si(\l) = D + \l C (I - \l A)^{-1} B.
 \]
Note that $\th_\si$ is an $\sL(\sU, \sY)$-valued function  which is defined and analytic  on  the open set consisting of all $\l \in \BC$ such that $I - \l A$ is boundedly invertible. In particular,  $\th_\si$ is  analytic in  an open neighborhood of zero.

With the system $\si = (A, B, C, D; \sX, \sU, \sY)$ we associate the  linear manifolds  $\im (A|B)$ and $\kr(C|A)$ which are defined as follows
\begin{equation}\label{contrspaces}
\im (A|B)=\linhull\, \{\im A^n B \mid {n\geq 0}\}, \quad
\kr(C|A) = \bigcap_{n \geq 0} \kr C A^n.
\end{equation}
Recall that  $\si$ is  \emph{minimal}  if  $\im(A|B)$ is dense in $\sX$  (i.e., $\si$ is \emph{controllable})   and  $\kr(C|A) =\{0\}$   ( i.e., $\si$ is \emph{observable}); cf.,  Theorem 2.1 in \cite{AKP06}. Finally, we denote by $M(\si)$ the \emph{system matrix associated with} $\si$, that is, $M(\si)$ is the $2\ts 2$ operator matrix defined by
\begin{equation}\label{systmat}
M(\si):=\begin{bmatrix}  A&B    \\ C &  D \end{bmatrix}:
\begin{bmatrix}  \sX    \\  \sU \end{bmatrix}\to
\begin{bmatrix}  \sX    \\  \sY \end{bmatrix}.
\end{equation}

In this paper we are interested in systems   that are \emph{passive} (or, in an other  terminology, \emph{dissipative})  with respect to the scattering supply rate function $w(u,y)=\|u \|^2-\|y \|^2$. The latter means  that   for each initial condition $x_0$ and each input sequence $u_0, u_1, u_2, \ldots$ we have
\[
\|x_{n+1}\|^2- \|x_{n}\|^2\leq \|u_{n}\|^2-\|y_{n}\|^2, \quad n=0,1,2, \dots,
\]
where $x_{n+1}$ and $y_n$ are determined from $u_n$ and $x_n$ via the system equations in \eqref{syst} In that case the   associate  system matrix is a contractive operator from $\sX\oplus\sU$ into  $\sX\oplus\sY$. The converse is also true. In other words, the system  $\si$ is passive if and only if   the operator $M_\si$ is a contraction. Moreover, in that case its transfer function $\th_\si$ is a Schur class function.

Our  main theorems given below concern   the Riccati equality and  Riccati  inequality   for discrete time systems with a scattering supply rate.   Analogous  results may be obtained for  other supply rates, e.g., impedance and transmission supply rates, and for continuous time systems. For these different supply rate functions see, e.g., the papers \cite{ArovNu96}, \cite{ArovNu00} and the references therein.

\begin{defi}\label{SolRE} Let $\si = (A, B, C, D; \sX, \sU, \sY)$. A $($possibly unbounded$)$
selfadjoint ope\-rator $H$ in $\sX$ is said to be a
\emph{generalized  solution of the Riccati equation associated to
$\si$} if the following four conditions are satisfied:
\begin{itemize}
\item [\textup{(C1)}] the operator $H$ is positive  as a selfadjoint operator, i.e, $\lg Hx,x\rg>0$ for  each
$0\not = x\in \sD(H)$;
\item [\textup{(C2)}] $A\sD(H^{1/2})\subset \sD(H^{1/2})$ and $B\sU \subset
\sD(H^{1/2})$;
\item [\textup{(C3)}]  the operator $\d_\si(H)= I_\sU-D^*D-(H^{1/2}B)^*H^{1/2}B$ is bounded and  nonnegative, and
\begin{equation} \label{incluC3}
\Big(D^*C+(H^{1/2}B)^*H^{1/2}A\Big)\sD(H^{1/2})\subset \d_\si(H)^{1/2}\sU;
\end{equation}
\item [\textup{(C4)}]  for each $x\in \sD(H^{1/2})$ we have

\begin{equation}\label{ric}
\begin{aligned}
\|H^{1/2}x\|^2 & - \|H^{1/2}Ax\|^2 - \|Cx\|^2
\\
& =\|\Big(\d_\si(H)^{1/2}\Big){}^{[-1]}\Big(D^*C+
(H^{1/2}B)^*H^{1/2}A\Big)x\|^2.
\end{aligned}
\end{equation}

\end{itemize}
\end{defi}
Here and in the sequel  $\sD(H)$ stands for the \emph{domain} of the operator $H$.   Since $H$ is a positive selfadjoint operator, we know from the theory of  operators (possibly unbounded)  on Hilbert spaces (see, e.g., Chapter XII in \cite{DSII63}) that  $H^{1/2}$ is well-defined and  a positive selfadjoint operator too.  Moreover,
\begin{align*}
\sD(H)&=\{x\in \sD(H^{1/2}) \mid H^{1/2}x\in  \sD(H^{1/2})\},\\
Hx&=H^{1/2}\big(H^{1/2}x\big)\quad \big(x\in \sD(H)\big).
\end {align*}
The latter two properties define  $H^{1/2}$ uniquely.

Note that (C1) and (C2) imply that the operator $H^{1/2}B$ is a bounded operator from $\sU$ into $\sX$, and the hence  the operator $\d_\si(H)$  defined in (C3) is automatically bounded.

The symbol $[-1]$ appearing in  the right hand side of the inequality \eqref{ricineq} means that the term involved is  the Moore-Penrose pseudo-inverse of   the nonnegative bounded operator $ \d_\si(H)^{1/2}$. See the final paragraph of the present section for the definition of this notion.    Note that $ \d_\si(H)$ can be a zero operator (see Theorem \ref{innerdeltaH}).

In what follows we refer to \eqref{ric} as the \emph{Riccati equality}  associated to $\si$. By $\textup{RE}_\si$ we shall denote the set of all   generalized  solutions $H$ of the Riccati equation associated to $\si$.  If $H\in \textup{RE}_\si$, then
\begin{equation}
\label{obs2} \im(A|B)\subset  \sD(H^{1/2}) \ands
\im(A^*|C^*)\subset  \sD(H^{-1/2}).
\end{equation}
The first inclusion follows from condition (C2). The second inclusion in \eqref{obs2} requires a proof which will be given in  the next section; see Lemma \ref{lem:minsysH}.

By  $\textup{RE}_\si^\circ$ we denote the subset  of $\textup{RE}_\si$ consisting of all $H\in \textup{RE}_\si$ such
 that the  following two additional conditions are satisfied:
\begin{itemize}
\item [(a)] both $H^{1/2} \im(A|B)$ and  $(H^{-1/2}) \im (A^*|C^*)$ are   dense in $\sX$;
\item [(b)] the linear manifold  $ \im(A|B) $ is a core for the  operator $H^{1/2}$.
\end{itemize}
By definition (see, e.g., Section III.5.2 in \cite{Kato66})
condition (b) means that the  linear manifold $\{(u, H^{1/2}u)
\mid   u\in \im(A|B)\}$ is dense in the graph of $H^{1/2}$ with
respect to the graph norm. Note that the sets $H^{1/2} \im(A|B)$
and  $(H^{-1/2}) \im (A^*|C^*)$ are well defined because of
\eqref{obs2}. For a better understanding of condition (a) we refer
to Lemma  \ref{lem:minsysH} in Section~\ref{sec:Hpass} below. We
shall prove the following theorems.

\begin{thm}\label{thm1a}
Let $\si = (A, B, C, D; \sX, \sU, \sY)$ be a minimal system. If  the set
$\textup{RE}_\si$ is nonempty, then  the transfer function
$\th_\si$ coincides with a Schur class function in a neighborhood
of zero.
\end{thm}

\begin{thm}\label{thm1b}
Let $\si = (A, B, C, D; \sX, \sU, \sY)$ be a minimal system,  and assume that
its transfer function  coincides with a Schur class function in a neighborhood
of zero. Then the set $\textup{RE}_\si^\circ$  is nonempty and this set
 contains a minimal element with respect to the usual partial ordering of  $($possibly unbounded$)$  nonnegative selfadjoint operators.
\end{thm}

 Let us recall (see \cite[page 330]{Kato66} or \cite[Section 5 ]{AKP06}) the definition of  the ordering referred to in the previous theorem.  Let $H_1$, $H_2$ be non-negative selfadjoint operators acting in a Hilbert space $\sX$. Then, by definition,  $H_1 \prec H_2$ means that
 \[
\sD(H_2^{1/2})\subset \sD(H_1^{1/2})\ands \| H_1^{1/2} x\|\leq  \| H_2^{1/2} x\|\quad \left(x\in \sD(H_2^{1/2})\right).
\]
If $H_1$ and $H_2$ are bounded, then $H_1 \prec H_2$ is equivalent to $H_1
\leq H_2$.

To prove the above two theorems it will be convenient first  to consider
 the Riccati inequality associated to $\si$. This  inequality appears
 when   the equality sign in   \eqref{ric}  is replaced by  a  ``greater  than
 equal to'' sign. In other words  condition (C4)  in Definition \ref{SolRE} is replaced
 by
\begin{itemize}
\item [\textup{(CI4)}]  \textsl{for each $x\in \sD(H^{1/2})$ we have}
\begin{equation}\label{ricineq}
 \begin{aligned}
\|H^{1/2}x\|^2 & - \|H^{1/2}Ax\|^2 - \|Cx\|^2
\\
&\geq\|\Big(\d_\si(H)^{1/2}\Big){}^{[-1]}\Big(D^*C+
(H^{1/2}B)^*H^{1/2}A\Big)x\|^2, \quad x\in\sD(H^{1/2}).
\end{aligned}
\end{equation}
\end{itemize}
We shall say  that   a selfadjoint operator $H$ acting in $\sX$ is a
\emph{generalized  solution of the Riccati inequality associated to $\si$}
when conditions (C1), (C2), (C3), and (CI4)   are satisfied.
By $\textup{RI}_\si$ we shall denote the set of all
generalized  solutions $H$ of the Riccati inequality  associated to $\si$. Furthermore,  $\textup{RI}_\si^\circ$ will
denote the subset  of $\textup{RI}_\si$ consisting of all
$H\in \textup{RI}_\si$ such the two additional conditions (a) and (b)
above are satisfied.  Clearly, the following inclusions hold:
\begin{equation}\label{inclu}
\textup{RE}_\si\subset \textup{RI}_\si, \quad
\textup{RE}_\si^\circ\subset \textup{RI}_\si^\circ.
\end{equation}
These inclusions will allow us to derive Theorems \ref{thm1a} and  \ref{thm1b}
as corollaries of the following two results.

\begin{thm}\label{thm2a}
Let $\si = (A, B, C, D; \sX, \sU, \sY)$ be a minimal system. Then the set
$\textup{RI}_\si$ is nonemepty  if and only  if the transfer
function of $\th_\si$ coincides with a Schur class function in a
neighborhood of zero.
\end{thm}

\begin{thm}\label{thm2b}
Let $\si = (A, B, C, D; \sX, \sU, \sY)$ be a minimal system,  and assume that
its transfer function  coincides with a Schur class function in a neighborhood
of zero. Then the set $\textup{RI}_\si^\circ$  is nonempty and this set
contains a minimal element $H_\circ$ and a maximal element $H_\bullet$
with respect to the usual ordering  of  nonnegative operators. Furthermore, the minimal element $H_\circ$ in $\textup{RI}_\si^\circ$ belongs to the set
$\textup{RE}_\si^\circ$.
\end{thm}

In Section \ref{sec:KYP} we shall show that the Riccati inequality
is closely related to the Kalman--Yakubovich--Popov inequality.
This allows us to prove  (see the first paragraph after Theorem
\ref{thm:RIKYP} in Section \ref{sec:KYP}) that Theorem \ref{thm2a}
is equivalent to Theorem 1.2  in  \cite{AKP06}, and that  Theorem
\ref{thm2b},   except for its  final statement,  is equivalent to
Theorem~5.1   in  \cite{AKP06}. The final statement of  Theorem
\ref{thm2b} will be  proved in Section \ref{sec:finalthm2.5}.

As we mentioned,  Theorems \ref{thm1a} and  \ref{thm1b} appear as corollaries of Theorems \ref{thm2a} and  \ref{thm2b}.  Indeed,   if  $\textup{RE}_\si$ is nonempty, then  the same holds true for  $\textup{RI}_\si$ because of the first inclusion in \eqref{inclu}.   But then Theorem \ref{thm2a} tells us that $\th_\si$ coincides   with a Schur class function in a neighborhood of zero, which proves   Theorem \ref{thm1a}.  Thus Theorem \ref{thm1a} is covered by   the ``if part'' of Theorem \ref{thm2a}. In a similar way, using the   second inclusion  in \eqref{inclu} and the final statement of    Theorem \ref{thm2b},   one sees that   Theorem \ref{thm1b} is   covered  by Theorem \ref{thm2b}.

The paper consists of seven sections including the present introduction and an appendix.  In Section \ref{sec:Hpass}   the set $\textup{RI}_\si$ is related to the set of  $H$-passive systems. Furthermore, given $H\in \textup{RI}_\si$ we give a necessary and sufficient condition on $H$ in order that $H\in \textup{RE}_\si$.  In Section \ref{sec:KYP} we make explicit the relation between the Riccati inequality   and the Kalman-Yakubovic-Popov inequality which allows us to show  that Theorem~\ref{thm2a} is equivalent to Theorem~1.2. in \cite{AKP06} and Theorem \ref{thm2b} (except for the final statement)  is equivalent to Theorem 5.1 in \cite{AKP06}.  The final statement in Theorem \ref{thm2b} is proved in Section \ref{sec:finalthm2.5}. In Section \ref{sec:uniquesol}, using the last part of  Theorem 7.1 in \cite{AKP06}, we present a necessary and sufficient condition  for $\textup{RI}_\si^\circ$ to consist  of  a single element only, and we specify this  result   for the case when $\th$ is an inner or a co-inner function.  Examples illustrating the general theory are given in Section \ref{sec:examples}. In the Appendix we  review a number of results regarding $2\ts 2$ nonnegative operator matrices and related Schur complements  that are used in the present paper.

 \medskip

\noindent\textbf{Moore--Penrose pseudo-inverse.}  Let $A$ be a
bounded selfadjoint operator on a Hilbert space $\sX$. Put
$\sX_1=\overline{A\sX}$ and $\sX_2=\sX\ominus \sX_1$. Since $A$ is
selfadjoint, $\sX_2$ is the  null space of $A$. It follows that
relative to the Hilbert space  orthogonal direct sum
$\sX=\sX_1\oplus \sX_2$ the operator $A$ has the following $2 \ts
2$ operator matrix representation:
\begin{equation}\label{matA}
A=\begin{bmatrix}A_1&0\\ 0&0  \end{bmatrix}: \begin{bmatrix}\sX_1\\\sX_2  \end{bmatrix}\to \begin{bmatrix}\sX_1\\ \sX_2  \end{bmatrix}.
\end{equation}
The fact  that $\sX_2$  is the null space of $A$, implies that the operator $A_1$ maps $\sX_1$ in  one-to-one way into itself and $A_1\sX_1$ is equal to the range of $A$ which is dense in $\sX_1$.  By $A^{[-1]}$ we denote the closed linear  operator given by
\[
A^{[-1]}=\begin{bmatrix}A_1^{-1} &0\\ 0&0  \end{bmatrix}:
 \begin{bmatrix}\im A_1\\ \sX_2  \end{bmatrix}\to \begin{bmatrix}\sX_1\\ \sX_2  \end{bmatrix}.
\]
We call  $A^{[-1]}$ the Moore--Penrose pseudo-inverse of  $A$. Its
domain $\sD(A^{[-1]})$ is the \break linear space $\im A_1\oplus
\sX_2$. Note that $A^{[-1]}$ is  a selfadjoint operator, possibly
unbounded. Furthermore,  $A^{[-1]}$ is a zero  operator if and
only  if $ A$ is a zero operator.

Now assume that $A$ is a bounded selfadjoint operator on $\sX$  which is nonnegative. Then $A^{[-1]}$ is nonnegative too, and  the square roots $A^{1/2}$  and $(A^{[-1]})^{1/2}$ are well-defined. Note that the spaces $\overline{A\sX}$  and $\overline{A^{1/2}\sX}$ coincide. Using the latter and the operator matrix representation \eqref{matA}, it is not difficult to show  that
\begin{equation}\label{sqrts}
(A^{1/2})^{[-1]}=(A^{[-1]})^{1/2}.
\end{equation}
In particular, these two operators have the same domain.

\setcounter{equation}{0}
\section{The set $\textup{RI}_\si$  and related $H$-passive  systems}\label{sec:Hpass}

Let $\si = (A, B, C, D; \sX, \sU, \sY)$  be a   linear discrete time-invariant  system, and let $H\in \textup{RI}_\si$.   Since $H$ is a positive operator, the same is true for $H^{1/2}$,  and  both $H$ and $H^{1/2}$ are one-to-one.  It follows (cf., the first paragraph of Subsection 4.1 in \cite{AKP06}) that the following   operators are well defined:
\begin{align}
&A_H: \im H^{1/2}\to \sX, \quad A_H H^{1/2}x=H^{1/2}Ax \quad \big(x\in \sD(H^{1/2})\big),\label{defAH}\\
&C_H: \im H^{1/2}\to \sY, \quad C_H H^{1/2}x=Cx \quad \big(x\in \sD(H^{1/2})\big),\label{defCH}\\
& B_H: \sU\to \sX, \quad B_H u=H^{1/2}Bu\quad (u\in \sU).\label{defBH}
\end{align}
From condition (CI4) we see that
\[
\| z\|^2- \| A_H z\|^2 - \| C_H z\|^2\geq 0 \quad (z\in \im H^{1/2}).
\]
Thus $A_H$ and $C_H$ are bounded in norm by one on $\im H^{1/2}$. Since $\im H^{1/2}$ is dense in $\sX$, we can extend  $A_H$ and $C_H$ by continuity to contractions on $\sX$ which also will  be denoted by  $A_H$ and $C_H$. From the second part of condition (C2) it follows  that  $B_H$ is well defined bounded operator, and the first  part of condition (C3) implies that $B_H$ is a contractive operator mapping  $\sU$ into $\sX$. Put
\begin{equation}
\label{defsiH}
\si_H=(A_H, B_H, C_H, D; \sX, \sU, \sY).
\end{equation}
We shall call $\si_H$ \emph{the system  associated with $\si$ and  $H$}. Recall  that the system matrix $M(\si_H)$ associated with $\si_H$ is given by
\[
M(\si_H)=\begin{bmatrix} A_H& B_H\\ C_H & D\end{bmatrix}:
\begin{bmatrix} \sX\\ \sU \end{bmatrix} \to
\begin{bmatrix} \sX\\  \sY\end{bmatrix}.
\]
\begin{defi}
In the sequel   the system  $\si$ will be called \emph{$H$-passive} when $\si_H$ is passive. In other words,   $\si$ is   $H$-passive if and only if  $M(\si_\sH)$ is contractive.
\end{defi}

\begin{thm}\label{contracMRI} Let $H\in \textup{RI}_\si$. Then the system $\si$ is $H$-passive. Furthermore,   $H\in \textup{RE}_\si$ if and only if
\begin{equation}
\label{infsystM0}
\inf \left\{\|\begin{bmatrix} x\\u\end{bmatrix}\|^2-
 \|M(\si_H)\begin{bmatrix} x\\u\end{bmatrix} \|^2 \mid  u\in \sU \right\}= 0 \quad (x\in \sX).
\end{equation}
\end{thm}
\begin{proof}
We split the proof into two parts. First we show that the system
$\si$ is $H$-passive.

\smallskip
\noindent\textsc{Part 1.} Using the definitions in \eqref{defAH}, \eqref{defCH}, and \eqref{defBH} we see that condition (C3) can be rephrased as
\begin{itemize}
\item [\textup{(C3')}] the operator $\d_\si(H) =I_\sU-D^*D-B_H^*B_H$ is bounded and nonnegative, and
\begin{equation}\label{condC3a}
 (D^*C_H+B_H^*A_H)z\in \d_\si(H)^{1/2}\sU\quad (z\in \im H^{1/2}).
\end{equation}
\end{itemize}
Similarly,  (CI4) can be rephrased as
\begin{equation}\label{condCI4a}
\begin{aligned}
\|z\|^2-\|A_Hz\|^2 & -\|C_H z\|^2
\\
& \geq
\|\left(\d_\si(H)^{1/2}\right){}^{[-1]}(D^*C_H+B_H^*A_H)z
\|^2\quad (z\in \im H^{1/2}).
\end{aligned}
\end{equation} Next, put
\[
\a =I_\sX-A_H^*A_H - C_H^*C_H, \quad \b=-A_H^*B_H-C_H^*D, \quad \d=\d_\si(H).
\]
Then
\begin{equation}
\label{defR}
R:=I_{\sX\oplus \sU}- M(\si_H)^*M(\si_H)=
\begin{bmatrix} \a    &  \b  \\  \b^*    &  \d \end{bmatrix}.
\end{equation}
In order to prove that the system $\si$ is $H$-passive   we have to show that the $2\ts 2$ operator matrix in the right hand side of \eqref{defR} is nonnegative. To do this we apply Proposition ~\ref{prop:A1}.

Note that
$$
\begin{aligned}
\lg \a z, z\rg & = \lg z, z\rg- \lg A_H^*A_Hz, z\rg- \lg
C_H^*C_Hz, z\rg
\\
& =\| z\|^2- \| A_H z\|^2 - \| C_H z\|^2\geq 0 \quad
(z\in \im H^{1/2}).
\end{aligned}
$$
Since $\im H^{1/2}$ is dense in $\sX$ and   the operators $A_H$
and $C_H$ are bounded, the preceding inequality shows, by
continuity, that
\[
\lg \a x, x\rg=\| x\|^2- \| A_H x\|^2 - \| C_H x\|^2\geq 0 \quad (x\in \sX).
\]
Hence $\a\geq 0$. We already know that $\d=\d_\si(H)$ is nonnegative too. Next, note that  \eqref{condC3a} and \eqref{condCI4a} yield
\begin{align}
&\b^*z\in\d^{1/2}\sU \quad (z\in \im H^{1/2}), \label{condC3b} \\
&\lg \a z, z\rg\geq \|( \d^{1/2})^{[-1]} \b^*z  \|^2 \quad (z\in
\im H^{1/2}).  \label{condCI4b}
\end{align}

Recall that $\d$ is bounded and nonnegative. Thus $\d_0:=\d|_{\overline{\im \d}}$ is a one-to-one operator on $\overline{\im \d}$ and the range $\im \d_0$  is dense in $\overline{\im \d}$. Since
\[
(\d^{1/2})^{[-1]}=(\d^{[-1]})^{1/2}=\begin{bmatrix}\d_0^{-1/2}&0\\0&0\end{bmatrix}:
\begin{bmatrix}\overline{\im \d}\\ \kr \d \end{bmatrix} \to
\begin{bmatrix}\overline{\im \d}\\ \kr \d \end{bmatrix},
\]
we conclude that the range of $(\d^{1/2})^{[-1]}$ is a subset of  $\overline{\im \d}$. Now define
\begin{equation}\label{defGa0}
\begin{aligned}
&\ga_0: \a^{1/2}( \im H^{1/2})\to \overline{\im \d},
\\
&  \ga_0 (\a^{1/2}z)= (\d^{1/2})^{[-1]} \b^*z,  \quad  z\in \im H^{1/2}.
\end{aligned}
\end{equation}
According to the identity \eqref{condCI4b} the
operator $\ga_0$ is well defined and $\ga_0$ is a contraction.
Observe that
\[
\overline{ \a^{1/2}( \im H^{1/2})}= \overline{ \a^{1/2}\left( \overline{\im H^{1/2}}\right)}=\overline{ \a^{1/2}\sX}=\overline{ \a\sX}= \overline{\im  \a}.
\]
But then, by continuity, the contraction $\ga_0$ extends to a contraction $\wt{\ga_0}$ mapping $ \overline{\im  \a}$ into  $ \overline{\im  \d}$ and such that
\[
\d^{1/2}\wt{\ga_0} \a^{1/2}z= \d^{1/2}\ga_0 \a^{1/2}z=\d^{1/2}(\d^{1/2})^{[-1]} \b^*z=\b^*z\quad (z\in \im H^{1/2}).
\]
Here we used that $\d^{1/2}(\d^{1/2})^{[-1]}$ is the orthogonal projection onto $\overline{\im \d^{1/2}}$ and the fact that $\im \b^*\subset \overline{\im \d^{1/2}}$ which follows from \eqref{condC3b}.  Since $\im H^{1/2}$ is dense in $\sX$ and the operators  $\d^{1/2}\wt{\ga_0} \a^{1/2}$  and $\b^*$ are bounded operators, we conclude, by continuity, that $\b^*=\d^{1/2}\wt{\ga_0} \a^{1/2}$.  Finally, define $\ga:\sX\to \sU$ by
\begin{equation}
\label{defGa1}
\ga|_{\overline{\im  \a}}= \wt{\ga_0} \ands\ga|_{\kr \a}= 0.
\end{equation}
Then $\ga :\sX\to \sU$ is a contraction satisfying conditions (a) and (b) in Proposition
 \ref{prop:A1}, and hence  we can apply Proposition  \ref{prop:A1} with $T=R$ to show that the operator $R$
 in \eqref{defR} is nonnegative. Hence $M(\si_H)$ is a contraction, and the first part of the proposition  is proved.

\smallskip

\noindent\textsc{Part 2.} In this part given $H\in \textup{RI}_\si$ we show that $H\in \textup{RE}_\si$ if and only if \eqref{infsystM0} holds.   Since $H\in \textup{RI}_\si$ we can freely use the operators introduced in the previous part. In particular, $R$ is the operator defined by \eqref{defR} and $\ga$ is the contraction defined by \eqref{defGa1}.

First we assume that $H\in \textup{RE}_\si$. This  implies (see condition (C4)) that we have equality in \eqref{condCI4a} and in \eqref{condCI4b}, and hence the operator $\ga_0$ defined in \eqref{defGa0} is an isometry. But then,  following the reasoning in the previous part of the proof,  we see that $\wt{\ga_0}$, the continuous extension of ${\ga_0}$ to $\overline{\im \a}$, is an isometry too, and thus the operator $\ga$ defined by \eqref{defGa1} is a partial isometry with initial space $\im \a$. But then the Schur complement $Z=\a^{1/2}(I-\ga ^*\ga)\a^{1/2}$ is the zero operator, and we can apply Proposition \ref{prop:A2} to show that \eqref{infsystM0} holds.

The converse implication follows in a similar way reversing the arguments. Indeed, assume  \eqref{infsystM0} holds.  Then Proposition \ref{prop:A2} tells us that the Schur complement of  $R$   supported by $\sX$ is equal to zero. Here $R$ is given by \eqref{defR}. Thus $\a^{1/2}(I-\ga^*\ga)\a^{1/2}=0$, where $\ga$ is the minimal contraction determined by $R$, which in our case is  the contraction defined by \eqref{defGa1}.  Thus  $\ga$ is a partial isometry with initial space $\overline{\im \a}$. It follows that  $\ga_0$ defined by \eqref{defGa0}  also is  an isometry.  But then we have equality in \eqref{condCI4b} and hence also in \eqref{condCI4a}, Thus condition (C4) is satisfied which implies that $H\in \textup{RE}_\si$.
\end{proof}

We conclude this section with the following lemma. For the definition of the notion of pseudo-similarity we refer to
\cite[Section 3]{AKP06}.

\begin{lem}\label{lem:minsysH} Let $\si = (A, B, C, D; \sX, \sU, \sY)$ be a minimal system, and let  $H\in \textup{RI}_\si$. Then the systems $\si$ and $\si_H$ are pseudo-similar and $H^{1/2}$ is a pseudo-similarity from $\si$ to $\si_H$. Furthermore,  the inclusions in \eqref{obs2} are satisfied, and $\si_H$ is minimal if and only  if   both $H^{1/2} \im(A|B)$ and  $(H^{1/2})^{-1} \im (A^*|C^*)$ are   dense in $\sX$.
\end{lem}
\begin{proof}   Put $S=H^{1/2}$. To prove that $S$ is a
pseudo-similarity from $\si$ to $\si_H$ we have to check (see
formulas (3.1)--(3.4) in \cite{AKP06}) the following properties:
\begin{align}
\overline{\sD(S)}=\sX,   &\hspace{1cm} \overline{\im S}=\sX;
\label{pseusim1}
\\
A\sD(S)\subset \sD(S), &\hspace{1cm} SAx=A_HSx, \quad x\in\sD(S);
\label{pseusim2}
\\
B\sU \subset \sD(S), &\hspace{1cm} SB=B_H;\label{pseusim3}
\\
Cx=C_H Sx, \quad
 x\in\sD(S). & \label{pseusim4}
\end{align}
Condition (C1) in Definition \ref{SolRE} implies that $H^{1/2}(\sX\to \sX)$ is a closed, injective,  densely defined operator, and its  range is dense in $\sX$.  Since $S=H^{1/2}$, it follows that   \eqref{pseusim1} holds. Formulas \eqref{pseusim2} and \eqref{pseusim3} follow  from condition (C2) in Definition \ref{SolRE} using the definitions of $A_H$ and $B_H$ in \eqref{defAH} and \eqref{defBH}, respectively.  Formula  \eqref{pseusim4} follows from the definition of $C_H$ in \eqref{defCH}. Thus $S=H^{1/2}$ is a pseudo-similarity from $\si$ to $\si_H$.

The identities in the right hand side of  \eqref{pseusim2} and \eqref{pseusim3} tell us that $\im (A|B)$ is a subset of  $\sD(H^{1/2})$. Thus the the first inclusion in  \eqref{obs2} holds true. Furthermore, we have
\begin{align*}
\im (A_H|B_H)&= \linhull\, \{\im A_H^nB_H\mid  n=0,1,2, \ldots\} \\
&=\linhull\,\{\im H^{1/2}A^nB\mid  n=0,1,2, \ldots\} =H^{1/2}\im (A |B).
\end{align*}
This implies that $\si_H$ is controllable if and only if  $H^{1/2}\im (A |B)$ is dense in~$\sX$.

Next we apply the final part of Proposition 3.1 in \cite{AKP06}. It follows that $S^{-1}=H^{-1/2}$  is a pseudo--similarity from $\si_H$ to $\si$. But then $(S^{-1})^*=H^{-1/2}$ is a   pseudo--similarity from $\si^*$ to$ (\si_H)^* $, where
\begin{align}
\si^*&= (A^*, C^*, B^*, D^*; \sX, \sY, \sU), \label{adjointsyst}\\
(\si_H)^*& =(A_H^*, C_H^*, B_H^*, D^*; \sX, \sY, \sU). \label{adjointsystH}
\end{align}
In particular, using  \eqref{pseusim2} and   \eqref{pseusim3},  we have
\begin{align*}
A^*\sD(H^{-1/2})\subset \sD(H^{-1/2}),&\qquad
H^{-1/2}A^*x=A_H^*H^{-1/2}x, \quad x\in \sD(H^{-1/2});\\
C^*\sY\subset \sD(H^{-1/2}),&\qquad H^{-1/2}C^*=C_H^*.
\end{align*}
Thus  $\im (A^*|C^*) \subset  \sD(H^{-1/2})$, and hence the second inclusion in  \eqref{obs2} holds true. Furthermore, using the same calculation  for $A_H^*, C_H^*$  as for  $A_H, B_H$ in the previous paragraph, we obtain   $\im (A_H^*|C_H^*)=H^{-1/2}\im (A^*|C^*)$, which shows that  $\si_H$ is observable if and only if  the space $H^{-1/2}\im (A^*|C^*)$ is dense in $\sX$. This completes the proof.
\end{proof}
The system $\si^*$ defined by \eqref{adjointsyst} is called the \emph{adjoint} of  the system $\si$.  Using the main results of the next section we shall derive some further properties of the adjoint system  at the end of Section \ref{sec:finalthm2.5}.

\setcounter{equation}{0}
\section{The Kalman--Yakubovich--Popov inequality}\label{sec:KYP} The Riccati inequality is closely related to the
Kalman--Yakubovich--Popov inequality (for short, KYP inequality).
Recall (see Section 1 of \cite{AKP06})  that a  (possibly
unbounded)  selfadjoint operator $H$ acting in $\sX$ is called a
\emph{generalized solution  of the KYP  inequality associated to
$\si$} if conditions \textup{(C1)}  and \textup{(C2)} are
satisfied, and
\begin{equation}\label{condKsi}
K_\si(H)\begin{bmatrix} x\\ u\end{bmatrix} \geq 0, \quad x \in
\sD(H^{1/2}), \quad u \in \sU,
\end{equation}
where
\begin{align}
&K_\si(H) \begin{bmatrix} x\\ u\end{bmatrix} =
\big\|\begin{bmatrix}H^{1/2} & 0 \\ 0 & I_\sU \end{bmatrix}
\begin{bmatrix}  x \\ u \end{bmatrix}\big\|^2
-\big\|\begin{bmatrix} H^{1/2} & 0
\\ 0 & I_\sY \end{bmatrix}
\begin{bmatrix}  A & B \\ C & D \end{bmatrix}
\begin{bmatrix}  x \\ u \end{bmatrix}\big\|^2.\label{defKsi}
\end{align}
 Note that condition (C2) tells us that $Ax+Bu\in \sD(H^{1/2})$ whenever $x$ and $u$ are as in \eqref{condKsi}. Thus $K_\si(H)$ is well defined.    See \cite{ArovSt07}  for continuous time analogues  of the results in~\cite{AKP06}.

In what follows $\textup{KYP}_\si$ denotes the set of all generalized solution  of the KYP inequality associated to $\si$.  We shall prove the following theorem.

\begin{thm}\label{thm:RIKYP}   A   selfadjoint operator $H$ acting on $\sX$ belongs to $\textup{RI}_\si$ if and only if $H$ belongs to  $\textup{KYP}_\si$, that is,  $\textup{RI}_\si=\textup{KYP}_\si$.
\end{thm}

When  Theorem \ref{thm:RIKYP} is proved, then Theorem \ref{thm2a}
is proved too. In fact, if  Theorem \ref{thm:RIKYP} is proved,
then Theorem \ref{thm2a} is equivalent to Theorem 1.2. in
\cite{AKP06}. Analogously, Theorem \ref{thm2b} (except for the
final sentence)  is equivalent to Theorem 5.1 in \cite{AKP06}. The
statement in the final sentence of Theorem \ref{thm2b}  will be
proved in the next section.

We shall denote by $\textup{KYP}_\si^\circ$ the set of all $H$ in
$\textup{KYP}_\si$ that satisfy the additional conditions (a) and
(b) appearing in the paragraph preceding Theorem~\ref{thm1a}. From
Lemma~\ref{lem:minsysH}  we know that  condition  (a)  just means
that $\si_H$ is minimal. It follows that $\textup{KYP}_\si^\circ$
coincides with the set  which in \cite{AKP06} is denoted by
$\sG\sK_{\si, \textup{core}}^{\textup{min}}$; see \cite[formulas
(5.1) and (5.2)]{AKP06}.  Using Theorem \ref{thm:RIKYP} and the
definitions of the sets $\textup{RI}_\si^\circ$ and
$\textup{KYP}_\si^\circ$ we obtain the following corollary.

\begin{cor}\label{cor:RIKYPcirc} A   selfadjoint operator $H$ acting on $\sX$ belongs to $\textup{RI}_\si^\circ$ if and only if $H$ belongs to  $\textup{KYP}_\si^\circ$, that is,  $\textup{RI}_\si^\circ=\textup{KYP}_\si^\circ$.
\end{cor}

In order to prove  Theorem \ref{thm:RIKYP} we need some  preliminaries. Assume that $H\in \textup{KYP}_\si$. By specifying \eqref{condKsi} for the vectors  $(x,0)$ and $(0,u)$ we see that
\begin{align*}
&\|H^{1/2} x \|^2 - \|H^{1/2} A x \|^2 - \|C x\|^2 \geq 0 \quad \big(x\in \sD(H^{1/2})\big), \\[.2cm]
&\|u\|^2 - \|D u \|^2 - \|H^{1/2} B u \|^2  \geq 0 \quad  (u\in \sU).
\end{align*}
As we proved in Subsection 4.1  of  \cite{AKP06}, this allows one to define  operators $A_H$, $B_H$ and $C_H$ in the same way as in the paragraphs preceding Proposition  \ref{contracMRI}. Also in this setting the resulting system $\si_H$, defined as in \eqref{defsiH}, is  called the \emph{the system  associated  with $\si$ and  $H$}. The following lemma, which is the analogue of  the first part of  Proposition  \ref{contracMRI}  with $H\in \textup{KYP}_\si$ in place of  $H\in \textup{RI}_\si$, is covered by Proposition 4.2 in \cite{AKP06}.

\begin{lem}\label{contrMKYP}Let $H\in \textup{KYP}_\si$. Then the system $\si_H$ associated with $\si$ and  $H$ is passive.
\end{lem}

\noindent\textit{Proof of Theorem \ref{thm:RIKYP}}. We split the
proof into to parts. In the first part we show that $H\in
\textup{RI}_\si$ implies that  $H\in \textup{KYP}_\si$. The second
part proves the reverse implication.

\smallskip
\noindent\textsc{Part 1.}  Let $H\in \textup{RI}_\si$, and let $\si_H=(A_H, B_H, C_H, D; \sX, \sU, \sY)$ be the system associated  with $\si$ and  $H\in \textup{RI}_\si$.  In particular, $H$ satisfies conditions (C1) and (C2). Thus it remains to prove \eqref{condKsi}. In order to that, fix $x\in \sD(H^{1/2})$ and $u \in \sU$. Then
\begin{align*}
K_\si(H) \begin{bmatrix}x\\u\end{bmatrix}&= \big\|\begin{bmatrix}H^{1/2} & 0 \\ 0 & I_\sU \end{bmatrix}
\begin{bmatrix} x\\ u\end{bmatrix}\big\|^2-
\big\|\begin{bmatrix}H^{1/2} & 0 \\ 0 & I_\sY \end{bmatrix}
\begin{bmatrix} A& B\\ C& D\end{bmatrix}
\begin{bmatrix} x\\ u\end{bmatrix}\big\|^2 \\
&=\big\|\begin{bmatrix} H^{1/2}x\\ u \end{bmatrix}\big\|^2-
\big\|\begin{bmatrix}H^{1/2}A&H^{1/2}B\\ C& D\end{bmatrix}\begin{bmatrix} x\\ u \end{bmatrix}\big\|^2\\
&=\big\|\begin{bmatrix} H^{1/2}x\\ u \end{bmatrix}\big\|^2-
\big\|\begin{bmatrix}A_H&B_H\\ C_H& D\end{bmatrix}\begin{bmatrix} H^{1/2}x\\ u \end{bmatrix}\big\|^2.
\end{align*}
But, by Theorem \ref{contracMRI}, the system matrix $M(\si_H)$  is a contraction. It follows that
\[
\big\|\begin{bmatrix}A_H&B_H\\ C_H& D\end{bmatrix}\begin{bmatrix} H^{1/2}x\\ u \end{bmatrix}\big\|\leq \big\|
\begin{bmatrix} H^{1/2}x\\ u \end{bmatrix}\big\|, \quad  x\in \sD(H^{1/2}), \quad u\in \sU.
\]
Thus \eqref{condKsi} holds true.

\smallskip
\noindent\textsc{Part 2.}
Let $H\in \textup{KYP}_\si$, and let $\si_H=(A_H, B_H, C_H, D; \sX, \sU, \sY)$ be the system associated  with $\si$ and  $H\in \textup{KYP}_\si$.   Since $H\in \textup{KYP}_\si$, we know that  conditions \textup{(C1)}  and \textup{(C2)} are satisfied.  It remains to check (C3) and (CI4).  According to Lemma \ref{contrMKYP}, the system matrix  $M(\si_H)$ is a contraction. This implies that the operator  $T$ defined by
\begin{equation}\label{poscond}
T=\begin{bmatrix}I_\sX-A_H^*A_H -C_H^* C_H&-A_H^*B_H-C_H^* D\\[.2cm]
-B_H^*A_H-D^*C_H &I_\sU-B_H^*B_H-D^*D \end{bmatrix}
\end{equation}
is a bounded nonnegative operator on the Hilbert space direct sum $\sX\oplus \sU$. This allows us to apply Proposition \ref{prop:A1} with
\begin{align}
&\a=I_\sX-A_H^*A_H -C_H^* C_H, \quad \b=-A_H^*B_H-C_H^* D, \label{abd1}\\
&\hspace{2cm} \d=I_\sU-B_H^*B_H-D^*D.\label{abd2}
\end{align}
Since $T$ defined by \eqref{poscond} is nonnegative,  Proposition \ref{prop:A1}  tells us that  $\a$ and $\d$ are nonnegative, and there exists a contraction $\ga$ mapping  $\sX$ into $\sU$ such that
\begin{equation}
\label{3cond}
\kr \ga \supset \kr \a, \quad \im \ga \subset \overline{\im \d}, \quad \b^*=\d^{1/2}\ga \a^{1/2}.
\end{equation}
Since $H^{1/2}B=B_H$ is a well-defined bounded operator (see \eqref{defBH}), we have
\[
\d_\si(H)= I_\sU-D^*D-(H^{1/2}B)^*H^{1/2}B=I_\sU-D^*D-B_H^*B_H =\d,
\]
and  hence $\d_\si(H)=\d$ is bounded  and nonnegative because $T$ given by  \eqref{poscond} is bounded  and nonnegative.  Furthermore, the  inclusion \eqref{incluC3} follows from the identity in the third part of \eqref{3cond}. To see this, note the equality  $\b^*=\d^{1/2}\ga \a^{1/2}$ implies that $\im \b^*\subset \im \d^{1/2}$.  Specifying this inclusion  for  $\b$ and $\d$  given by   \eqref{abd1} and \eqref{abd2}, respectively,   and using $\d_\si(H)=\d$ we obtain
\begin{align*}
\Big(D^*C &+(H^{1/2}B)^*H^{1/2}A\Big)\sD(H^{1/2})
\\
&
 =\Big (A_H^*B_H+C_H^*D\Big)\im H^{1/2} \subset \im
(A_H^*B_H+C_H^*D)\subset  \im \d_\si(H)^{1/2}.
\end{align*}
This proves  the inclusion \eqref{incluC3}.  Thus (C3) is satisfied.

It remains to prove the inequality \eqref{ricineq}. To do this we first observe that with  our choice of $H$, the inequality \eqref{ricineq} is equivalent to
\begin{align*}
\|z\|^2  -\|A_Hz\|^2 & -\|C_Hz\|^2
\\
& \geq \|(\d_\si(H)^{1/2})^{[-1]} (D^*C_H+B_H^*A_H)z\|^2, \quad
z\in \im H^{1/2}.
\end{align*}
Thus, using the two identities in \eqref{abd1}  and  $\d_\si(H)=\d$, in order to prove
\eqref{ricineq} we have to show that
\begin{equation}
\label{basicineq1}
\| \a^{1/2}z \|\geq \| (\d^{1/2})^{[-1]}\b^*z\|, \quad z\in \im H^{1/2}.
\end{equation}
But $\b^*=\d^{1/2}\ga \a^{1/2}$  yields $ (\d^{1/2})^{[-1]}\b^*
=\ga  \a^{1/2}$.  Since $\ga$ is a contraction,  we see that the
inequality in \eqref{basicineq1} holds for any $z\in \sX$. Thus
condition (CI4) is also satisfied.
{\hfill $\Box$}

\medskip

\begin{thm}The set $\textup{RI}_\si\not =\emptyset$ if and only if $\si$ is pseudo-similar to a passive system.
\end{thm}

\begin{proof}
We know (Theorem 4.1 in \cite{AKP06}) that  this is true for
$\textup{KYP}_\si$ in place of  $\textup{RI}_\si$. By Theorem
\ref{thm:RIKYP}, we have  $\textup{RI}_\si=\textup{KYP}_\si$.
Hence the result is also true for $\textup{RI}_\si$ in place of
$\textup{KYP}_\si$.
\end{proof}

\setcounter{equation}{0}
\section{Proof of the final statement in Theorem \ref{thm2b}}\label{sec:finalthm2.5}

The following proposition covers the final statement in Theorem \ref{thm2b}.

\begin{prop}\label{prop:minelt} Let $\si = (A, B, C, D; \sX, \sU, \sY)$ be a minimal system,  and assume that  its transfer function  coincides with a Schur class function in a neighborhood  of zero. If $H_\circ$ is a minimal element  in  $\textup{RI}_\si^\circ$ with respect to the usual ordering of nonnegative operators, then $H_\circ\in \textup{RE}_\si^\circ$.
\end{prop}

For the proof of the above proposition we  need  Lemma \ref{lem:minopt} below which is an addition to \cite[Theorem~5.1]{AKP97}. Recall (cf.,  Section 2 of \cite{AKP06}) that  a discrete time linear system  $\si$ is called a \emph{realization} of a Schur class function $\th$ whenever  the transfer function of $\si$  coincides with $\th$ in a neighborhood of zero. For the definition of an optimal passive system we refer to Section~3 in  \cite{AKP97}.

\begin{lem}\label{lem:minopt} Let $\si=(A,B,C,D, \sX, \sU, \sY)$  be a minimal and optimal passive  discrete time linear system.  Then
\begin{equation}\label{eq:minmatsi1}
\inf \left\{\|\begin{bmatrix} x\\u\end{bmatrix}\|^2-
 \|M(\si)\begin{bmatrix} x\\u\end{bmatrix} \|^2 \mid  u\in \sU \right\}= 0 \quad (x\in \sX).
\end{equation}
\end{lem}

\smallskip
The above lenma has been established in item (1) of
\cite[Corollary 7.3]{Arl08} using   results of M.~G.~ Kre\v{\i}n
on shorted operators; cf., the final paragraph of  the appendix
(Section \ref{App}). In the present paper we give a proof based on
the functional model of minimal passive optimal systems derived in
\cite{AKP06}.
\begin{proof}   Let $\si=(A,B,C,D, \sX, \sU, \sY)$  be a minimal and
optimal, and let $\th_\si$ be its transfer function.  Since $\si$
is passive, $\th_\si$ belongs to the Schur class $\sS(\sU, \sY)$,
that is, $\th_\si$ is analytic on the open unit disc $\BD$ and
$\|\th_\si(z)\|\leq 1$ for all $z \in \BD$.  This allows us to
replace $\si$ by its  restricted shift model. Indeed, let
$\th=\th_\si$, and let  $\si_\circ=(A_\circ,B_\circ,C_\circ,D,
\sX_\circ, \sU, \sY)$ be the minimal and optimal realization of
$\th$ given by Theorem 5.1 in \cite{AKP97}. Then $\si$ and
$\si_\circ$ are unitary equivalent  by Theorem 3.2 in
\cite{AKP97}, and hence it suffices to prove Lemma
\ref{lem:minopt} for $\si_\circ$ in place of $\si$.

Let us recall the construction of  $\si_\circ$ given in the paragraph preceding Theorem 5.1 in \cite{AKP97}. For this purpose we need the de Branges-Rovnyak space $\sH(\th):=\{f\in H^2(\sY)\mid \|f\|_{\sH(\th)}<\iy\}$, where $H^2(\sY)$ is the standard Hardy spaces of $\sY$-valued functions on the open unit disc $\BD$ with square summable Taylor coefficients and
\begin{equation}
\label{defnormHth}
\|f\|_{\sH(\th)}^2=\sup \{\|f+\th \eta\|_{H^2(\sY)}^2-\|\eta\|_{H^2(\sU)}^2\mid \eta\in H^2(\sU)\}.
\end{equation}
Let us list a few properties  (see, e.g.,   \cite[Chapter 2]{Ando90} and \cite[Section 2]{BB15})  of  the space $\sH(\th)$:
\begin{itemize}
\item[\textup{(a)}]  the space $\sH(\th)$ a Hilbert space with the Hilbert space norm  $\|\cdot\|_{\sH(\th)}$  being given by \eqref{defnormHth} and $\sH(\th)$ is contractively embedded in $H^2(\sY)$;
\item[\textup{(b)}]  the space $\sH(\th)$ is invariant under the backward-shift operator on $H^2(\sY)$, that is, if $f\in \sH(\th)$, then the function $\tilde{f}$, $\tilde{f}(z)=z^{-1}\left(f(z)-f(0)\right)$, also belongs to $\sH(\th)$;
\item[\textup{(c)}]  for each $u\in \sU$ the function $\tilde{\th}(\cdot)u$, where $\tilde{\th}(z)=z^{-1}\left(\th(z)-\th(0)\right)$, belongs to $\sH(\th)$.
\end{itemize}
Furthermore, we need  the Hankel operator  $G_\th$ mapping $K^2(\sU)$ into $H^2(\sU)$. Here  $$K^2(\sU)=L^2(U)\ominus H^2(\sU),$$ and   for any separable Hilbert space $\sF$  we denote by $L^2(\sF)$  the Hilbert space of measurable  $\sF$-valued functions  $f$ on  the unit circle $\BT$   such that $\|f(\cdot)\|^2$ is Lebesgue integrable on  $\BT$,  and with the norm on $L^2(\sF)$ being defined  by
\[
\|f\|^2=\frac{1}{2\pi}\int_0^{2\pi} \|f(e^{it})\|^2\,dt.
\]
The action of  $G_\th$ is given by
\[
G_\th f= P_{H^2(\sY)}\th f, \quad  f\in K^2(\sU),
\]
where $P_{H^2(\sY)}$ is the orthogonal projection of $L^2(Y)$ onto $H^2(\sY)$.  Using the norm \eqref{defnormHth} and  items (b) and  (c)   above it follows (see, e.g., \cite[Lemma 5.2]{AKP97})  that the range of the Hankel operator  $G_\th$ is contained in the model space  $\sH(\th)$.

We are now ready to define the system $\si_\circ$. By definition, the state space $\sX_\circ$ is the closure of $\im G_\th$ in  $\sH(\th)$ and
\[
\begin{array}{lll}
A_\circ: \sX_\circ\to \sX_\circ,\quad  &(A_\circ x)(z)=z^{-1}(x(z)-x(0))\quad (x\in \sX_\circ);\\
B_\circ: \sU \to \sX_\circ, \quad  & (B_\circ u)(z)=z^{-1}(\th(z)-\th(0))u\quad (u\in \sU);\\
C_\circ: \sX_\circ\to \sY , \quad  &C_\circ x=x(0)\quad (x\in \sX_\circ);\\
D:\sU\to \sY, \quad  &Du=\th(0)u \quad (u\in \sU).
\end{array}
\]
These  operators are all well defined, and $\si_\circ=(A_\circ,B_\circ,C_\circ,D, \sX_\circ, \sU, \sY)$ is the minimal and optimal realization of  $\th$ given by Theorem 5.1 in \cite{AKP97}.

Now let us prove Lemma \ref{lem:minopt} with $\si_\circ$ in place of $\si$. Let  $\eta\in H^2(\sU)$. We decompose $\eta$ as  $\eta(z)=u+z \tilde{\eta}(z)$, where  $u=\eta(0)$ and $\tilde{\eta}(z)=z^{-1}\left(\eta(z)-\eta (0)\right)$. Note that the constant function $u$ and the function $z\tilde{\eta}(z)$ are perpendicular in $H^2(\sU)$, and thus
\begin{equation}\label{normeta}
\|\eta\|_{H^2(\sU)}^2=\|u\|^2+\|\tilde{\eta}\|_{H^2(\sU)}^2.
\end{equation}
Next observe that
\begin{align*}
(x+\th \eta)(z)  =x(0)+\th(0)\eta(0)&  +\big(x(z)-x(0)\big)
\\
&
+\big(\th(z)-\th(0)\big)\eta(0)+\th(z)\big(\eta(z)-\eta(0)\big).
\end{align*}
Furthermore, using $\eta(z)=u+z \tilde{\eta}(z)$ and the  definitions of the operators $A_\circ,B_\circ,C_\circ, D$ given above we see that
\begin{align*}
x(0)+\th(0)\eta(0)&=C_\circ x+Du,\quad x(z)-x(0) =z(A_\circ x)(z),\\[.1cm]
\big(\th(z)-\th(0)\big)\eta(0)&=z(B_\circ u)(z),\\[.1cm]
\th(z)\big(\eta(z)-\eta(0)\big)&=z(\th\tilde{\eta})(z), \quad z\in \BD.
\end{align*}

It follows  that
\begin{equation}
\label{normxthteta}
\|x+\th \eta\|_{H^2(\sY)}^2=\|C_\circ x+ Du\|^2+ \|A_\circ x  +B_\circ u  +\th\tilde{\eta}\|_{H^2(\sY)}^2.
\end{equation}
Using the identities \eqref{normeta} and \eqref{normxthteta} we see that
\begin{align*}
\|x+\th \eta\|_{H^2(\sY)}^2 & -\|\eta\|_{H^2(\sU)}^2
\\
& = \left( \|C_\circ x+ Du\|^2-\|u\|^2\right)+\left( \|A_\circ x
+B_\circ u
+\th\tilde{\eta}\|_{H^2(\sY)}^2-\|\tilde{\eta}\|_{H^2(\sU)}^2\right).
\end{align*}
But then, using the definition of the norm $\|\cdot \|_{\sH(\th)}$ in \eqref{defnormHth}, we obtain
\begin{align*}
\|x\|_{\sH(\th)}^2 &=\sup \Big{\{} \left( \|C_\circ x+
Du\|^2-\|u\|^2\right)
\\
&
\quad +\left( \|A_\circ x +B_\circ u
+\th\tilde{\eta}\|_{H^2(\sY)}^2-\|\tilde{\eta}\|_{H^2(\sU)}^2\right)
{\mid} u\in \sU, \ \tilde{\eta}\in H^2(\sU)  \Big{\}}
\\
&=\sup\Big{\{} \|C_\circ x+ Du\|^2+ \|A_\circ x  +B_\circ
u\|_{\sH(\th)}^2-\|u\|^2\mid u\in \sU\Big{\}}.
\end{align*}
We conclude that
\[
\inf \Big{\{}\|x\|_{\sH(\th)}^2+\|u\|^2-  \|M(\si_\circ)\begin{bmatrix} x\\u\end{bmatrix} \|_{\sH(\th)\oplus\, \sU}^2 \mid u\in \sU\Big{\}}=0.
\]
This proves the lemma for $\si_\circ$, and hence we are done.
\end{proof}

\medskip
\noindent\textit{Proof of  Proposition \ref{prop:minelt}}.  Let
$H_\circ$ be a minimal element  in  $\textup{RI}_\si^\circ$ with
respect to the usual ordering of nonnegative operators. It
suffices to show that $H_\circ\in \textup{RE}_\si$. Recall that
$\textup{RI}_\si^\circ=\textup{KYP}_\si^\circ$ by Corollary
\ref{cor:RIKYPcirc},  and that $\textup{KYP}_\si^\circ$ coincides
with the set $\sG\sK_{\si, \textup{core}}^{\textup{min}}$ used in
Section 5 of \cite{AKP06}; see the paragraph before Corollary
\ref{cor:RIKYPcirc}. These facts allow us to use  the final part
of item (ii) in \cite[Propositon 5.8]{AKP97}. It follows  that
$\si_{H_\circ}$ is a minimal and optimal  passive system. But then
we know from Lemma~\ref{lem:minopt} that equation
\eqref{eq:minmatsi1} holds with $\si_{H_\circ}$ in place of $\si$,
and  we can apply Proposition \ref{contracMRI} to conclude that
$H_\circ\in \textup{RE}_\si$.  {\hfill $\Box$}

\medskip The equalities  $RI_\si=KYP_\si$ and  $RI_\si^\circ=KYP_\si^\circ$, proved in Section  \ref{sec:KYP}, Theorem \ref{thm:RIKYP} and Corollary  \ref{cor:RIKYPcirc}, allow us to extend results proved in   Section  4 of \cite{AKP06} to   the setting considered in the present paper.  Among other things this provides  the following addition  to Theorem \ref{thm2b} for the  adjoint system.

\begin{thm}\label{thm:Ric2}
Let $\si=(A, B, C,  D; \sX, \sU, \sY)$  be a  discrete time-invariant system, and let $\si^* =(A^*, C^*, B^*, D^*; \sX, \sY, \sU)$ be its adjoint system. Then the transfer function of  $\si^* $  is given by $\th_{\si^*}=\th_\si^\sim$,  where $\th_\si^\sim (\l)=\th_\si(\bar{\l})^*$, and  $\si$ is minimal if and only if  $\si^*$ is minimal. Furthermore,  assuming $\si$ is minimal and  $\textup{RI}_\si$ is non-empty, we have
\begin{equation}
\label{sisi*}
\textup{RI}_{\si^*} ^\circ=\{H^{-1}\mid H \in  \textup{RI}_{\si}^\circ\}.
\end{equation}
Finally, if $H_\circ$ and $H_\bu$  are the minimal and maximal elements in $\textup{RI}_{\si} ^\circ$, then $H_\bu^{-1}$ and $H_\circ^{-1}$  are the minimal and maximal elements in $\textup{RI}_{\si^*} ^\circ$
\end{thm}

The analogue of  \eqref{sisi*} for the Riccati equality in place of the Riccati inequality, i.e.,  with $\textup{RI}$ replaced by $\textup{RE}$, does not hold. See \eqref{RicE4} in the final paragraph of Example \ref{ex:*optimal}.

\setcounter{equation}{0}
\section{A criterion for uniqueness  and inner functions}\label{sec:uniquesol}

Let  $\si$ be a minimal realization of  an inner function $\th$. In this section we show that in that case   $\textup{RI}_\si^\circ$ consists of  a single element, $H_\circ$ say, and  we prove that  $\d_\si(H_\circ)=0$.  Since $\textup{RI}_\si^\circ$ is equal to the set $\sG\sK_{\si, \textup{core}}^\textup{min}$  appearing in  \cite{AKP06}, we shall show that the  first statement can be obtained as  a corollary  of  the final part of Theorem 7.1 in \cite{AKP06}.  The second statement is proved in the second  part of this section.

Let us recall  the  final part of   \cite[Theorem 7.1]{AKP06}.   This  requires some preliminaries, which we take from \cite[pages 164, 165]{ArovNu00} with some minor changes.  Let $\th$ be an arbitrary function in $\sS(\sU, \sY)$, not necessarily inner.  It is known  \cite[Section V.4]{Sz-NFBK10} that there exist a Hilbert space $\sF_r\subset \sU$ and a  function $\va_r\in \sS(\sU, \sF_r)$ with the following  three  properties:
\begin{itemize}
\item[\textup{(a)}] $\va_r(z)^*\va_r(z)\leq I_\sU-\th(z)^*\th(z)$ for each $z\in \BD$;
\item[\textup{(b)}] for any  Schur class function $\va\in \sS(\sU, \sG)$, where $\sG$ is a Hilbert space, such that   $\va(z)^*\va(z)\leq I_\sU-\th(z)^*\th(z)$   for  each $z\in \BD$, we have
\[
 \va(z)^*\va(z)\leq \va_r(z)^*\va_r(z) \quad  {\text{for each}}  \quad   z\in
 \BD;
 \]
 \item [\textup{(c)}]  $\overline{\im \va_r(0) }= \sF_r$.
\end{itemize}
Here,  the inequalities are understood in the sense of  bounded selfadjoint operators on Hilbert  spaces. The function $\va_r$ can be normalized by the condition $\va_r(0)|_{\sF_r} $ is positive. With this additional normalization, the function $\va_r$ is uniquely defined  (see \cite{Sz-NFBK10}). From  \cite{Sz-NFBK10}  we also know that    properties  (a), (b), (c)  imply that  the function  $\va_r(z)$ is outer.

In a similar way one defines a maximal factor  $\va_l$ from the left. Indeed,  there exist a Hilbert space $\sF_l\subset \sY$ and a  function $\va_l\in \sS( \sF_l, \sY)$ with the following  three  properties:
\begin{itemize}
\item[\textup{(a')}] $\va_l(z)\va_l(z)^*\leq I_\sY-\th(z)\th(z)^*$ for each  $z\in \BD$;
\item[\textup{(b')}] for any  Schur class function $\psi\in \sS(\sG^\prime, \sY)$, where $\sG^\prime$ is a Hilbert space, such that   $\psi(z)\psi(z)^*\leq I_\sY-\th(z)\th(z)^*$ for each  $z\in \BD$, we have
\[
\psi(z)\psi(z)^*\leq  I_\sY-\th(z)\th(z) ^* \quad {\text{for
each}} \quad  z\in \BD;
\]
\item[\textup{(c')}]  $\overline{\im \va_l(0)^*}=\sF_l$.
\end{itemize}
In this case the function   $\va_l(\bar{z})^*$ is an  outer function, and normalization is obtained by requiring $\va_l(0)^*|_{\sF_l} $ to be a positive operator.

The functions  $\va_r$ and $\va_l$ are called  the \emph{right and
left  defect functions} of $\th$; see \cite[page~213]{ArlHdeS07}
and the references given therein.

Given $\th\in \sS(\sU, \sY)$ and the   defect functions $\va_r\in \sS(\sU, \sF_r)$ and $\va_l\in \sS(\sF_l, \sY)$, we know from \cite{BD97} that there exists a function $h_0$ in the space $L^\iy(\sF_l, \sF_r)$  of bounded measurable operator-valued  functions defined on the unit circle with  values in $\sL(\sF_l, \sF_r)$ such that the block operator matrix
\begin{equation}
\label{defh0}
\Theta(\z)=\begin{bmatrix}\va_l(\z) & \th(\z)\\  h_0(\z) & \va_r(\z)\end{bmatrix}:
\begin{bmatrix}\sF_l\\ \sU \end{bmatrix} \to  \begin{bmatrix}\sY\\ \sF_r \end{bmatrix}
\end{equation}
is contractive almost everywhere for  $\z \in \BT$. Moreover, according to \cite{BD97}, the operator function $h_0$  defined above is unique. We call $h_0$ the \emph{coupling function} defined  by $\th$.

Let $\th\in \sS(\sU, \sY)$, and let $h_0$ be  the coupling
function defined above.  Using \cite[Theorem~1.1]{ArovNu00} and
$\textup{RI}_\si^\circ=\sG\sK_{\si, \textup{core}}^\textup{min}$,
the final part of Theorem~7.1 in \cite{AKP06} yields the following
theorem.

\begin{thm}\label{single1} Let $\th\in \sS(\sU, \sY)$, and  let  $\si$ be  a minimal realization of  $\th$.   Then $\textup{RI}_\si^\circ$ consists of a single element if and only if  the following condition is satisfied:
\begin{itemize}
\item[\textup{(C)}] the coupling function $h_0$ defined by  $\th$   is the boundary value of a function from the Schur class $\sS(\sF_l, \sF_r)$.
\end{itemize}
Here $\sF_l$ and $ \sF_r$ are the Hilbert space appearing in \eqref{defh0}.
\end{thm}

Condition (C) above is item (iii) in \cite[Theorem 1.1]{ArovNu00}.  Note that condition (C) does not depend on the particular choice of the minimal system $\si$.  As the  proof of Theorem 7.1 in \cite{AKP06} shows,  Theorem \ref{single1} above  can be viewed as a corollary of  the equivalence of items (i) and (iii) in  \cite[Theorem 1.1]{ArovNu00}.

\begin{rem}\label{zerocoupl} If one of the spaces $\sF_l$ and $ \sF_r$ consists of the zero vector only, then the  coupling  function $h_0$ defined by $\th$   is zero. Hence  condition \textup{(C)}  is trivially satisfied and, by Theorem  \ref{single1},  the set $\textup{RI}_\si^\circ$ consists of  one element only for any minimal realization of~$\th$.
\end{rem}

\begin{cor}\label{scalaruniq} Let $\th$ be a scalar Schur class function. Then the  defect functions $\va_r$ and $\va_l$ coincide. Furthermore, $\va_r=\va_l=0$ if and only if  the function $\log(1-|\th(\cdot)|)$ is not Lebesgue integrable on the unit circle, and in that case the set $\textup{RI}_\si^\circ$ consists of  one element only for any minimal realization of $\th$.
\end{cor}
\begin{proof}
The fact that $\va_r$ and $\va_l$ coincide follows directly from
the fact that scalar functions commute.  Now assume that
$\va_r\not =0$. Then   $\log |\va_r(\cdot)|\in L^1(\BT)$;  see,
e.g.,  \cite[Theorem~1.2]{61Dev}.  Using  $\log
|\va(\cdot)|^2=2\log |\va(\cdot)|$ for any $\va$, we see that
$\log |\va_r(\cdot)|^2\in L^1(\BT)$. But then, since
\[ |\va_r(z)|^2\leq 1-|\th(z)|^2 \hspace{.3cm} (z\in \BD) \Longrightarrow \  |\va_r(\z)|^2\leq 1-|\th(\z)|^2 \hspace{.3cm}  (\z\in \BT\ a.e.),
\]
it follows  that $\log(1-|\th(\cdot)|^2)$ belongs to   $L^1(\BT)$. Next use
\[
\log(1-|\th(\cdot)|^2)= \log(1-|\th(\cdot)|) +\log(1+|\th(\cdot)|).
\]
The preceding  identity together with  the fact that $\log(1+|\th(\cdot)|)$ belongs to $L^1(\BT)$ shows that   $\log(1-|\th(\cdot)|)\in  L^1(\BT)$.

Conversely, assume that $\log(1-|\th(\cdot)|)\in  L^1(\BT)$. Then
the factorization problem $|\va(z)|^2\leq 1-|\th(z)|^2$ has a
nonzero solution $\va$ in $H^\iy$  by Theorem 1.2 in \cite{61Dev}
or Proposition~V.7.1 (b) in \cite{Sz-NFBK10}, and hence, $\va_r$
is not zero.

We conclude that  $\va_r  =0$ if and only if  $\log(1-|\th(\cdot)|)\not \in L^1(\BT)$. The final part of the corollary now follows directly from Remark  \ref{zerocoupl} above.
\end{proof}
\medskip

Now assume that $\th\in \sS(\sU, \sY)$ is inner. Then $I_\sU-\th(\z)^*\th(\z)=0$ almost everywhere for $\z\in \BT$, and hence the space  $\sF_r$ consists of the zero element only. Thus, by the above remark,   the set $\textup{RI}_\si^\circ$ consists of  one element only. This proves  the first part of the following theorem.

\begin{thm}\label{innerdeltaH} Let $\si$ be a minimal realization of  the inner function $\th\in \sS(\sU, \sY)$. Then $\textup{RI}_\si^\circ$ consists of a single element, $H_\circ$ say, and  $\d_\si(H_\circ)=0$.
\end{thm}
\begin{proof}
It remains to prove $\d_\si(H_\circ)=0$. Since the function $\th$
is inner, we know from the Sz-Nagy--Foias model theory
\cite{Sz-NFBK10} that  $\th$ has an observable realization
$$\si_1=(A_1, B_1, C_1, D; \sX_1, \sU, \sY)$$ such that
its system matrix $M(\si_1)$ is  unitary.  Now put $\sX_{10}=
\overline{\im (A_1 | B_1)}$. Relative  to the Hilbert space direct
sum $\sX_1=\sX_{10}\oplus \sX_{10}^\perp$  the operators $A_1,
B_1, C_1$ admit the following block matrix representations:
\begin{align}
&\hspace{1.8cm}A_1=\begin{bmatrix}A_{10}&\star\\  0&\star\end{bmatrix}:
\begin{bmatrix}\sX_{10}\\ \sX_{10}^\perp\end{bmatrix}\to
\begin{bmatrix}\sX_{10}\\ \sX_{10}^\perp\end{bmatrix}, \label{defAcirc}\\[.2cm]
&B_1=\begin{bmatrix}B_{10}\\ 0\end{bmatrix}: \sU\to
\begin{bmatrix}\sX_{10}\\ \sX_{10}^\perp\end{bmatrix}, \quad C_1=\begin{bmatrix}C_{10}&\star\end{bmatrix}:
\begin{bmatrix}\sX_{10}\\ \sX_{10}^\perp\end{bmatrix}\to \sY.\label{defBCcirc}
\end{align}
Put $\si_{10}=(A_{10}, B_{10}, C_{10}, D; \sX_{10}, \sU, \sY)$.  The above construction implies that  $\si_{10}$  is controllable. Furthermore, since $\si_{1}$  is observable, the same holds true for  $\si_{10}$. Thus $\si_{10}$ is a minimal system. Moreover , the transfer function of  $\si_{10}$ is equal to the transfer function of  $\si_{1}$. Thus $\si_{10}$ is a minimal realization of~$\th$.

Using the terminology  of Section 2.1 in \cite{AKP97}, the system  $\si_{10}$ is the first minimal restriction of  the system $\si$. But then,  by \cite[Theorem 3.2]{AKP97}, the system $\si_{10}$ is a minimal and optimal realization of $\th$.

We claim that $M(\si_{10})$ is an isometry. To see this note that
\begin{align*}
I_{\sX_1}-A_1^*A_1-C_1^*C_1&=
\begin{bmatrix}I_{\sX_{10}}-A_{10}^*A_{10}-C_{10}^* C_{10}&\star\\ \star &\star\end{bmatrix},\\
A_1^*B_1+C_1^*D&=\begin{bmatrix}A_{10}^* B_{10}+C_{10}^* D\\ \star \end{bmatrix},\\
I_\sU-B_1^*B_1 -D^*D&=I_\sU-B_{10}^*B_{10} -D^*D.
\end{align*}
Since $M_{\si_1}$ is unitary, the operators in the left hand side of the three identities above are all zero. Thus
\[
I_{\sX_{10}}-A_{10}^*A_{10}-C_{10}^* C_{10} =0, \quad A_{10}^*B_{10}+C_{10}^*D=0, \quad I_\sU-B_{10}^*B_{10} -D^*D=0.
\]
This shows that  $M(\si_{10})$ is an isometry.

Now use that   the systems $\si_{10}$ and $\si_{H_\circ}$ are
unitarily equivalent (see Theorem 3.2 in~\cite{AKP97}). It follows
that  $M(\si_{H_\circ})$  is an isometry which  implies that
$\d_\si(H_\circ)=0$.
\end{proof}

\medskip
\begin{rem}\label{rem:new1}  Using Proposition 4 in \cite{ArovNu02} and taking into account
 Theorem \ref{single1}, it can be shown that the two statements in Theorem \ref{innerdeltaH}
  remain true if the condition $\th$ is inner is replaced by the condition that the right defect
  function $\va_r$ of $\th$ is zero or, equivalently, that $\sF_r=\{0\}$.  In fact, with some minor
  changes the same proof can be used to derive this more general result.  Indeed, from the Sz-Nagy-Foias model
  theory we know that  $\th$ is the transfer function of  a simple  conservative realization $\si_1$.
  Here \emph{conservative} means  that the system matrix $M_{\si_1}$ is unitary.
  Furthermore,  it is known \textup{(}item \textup{(a)} in \cite[Proposition~4]{ArovNu02}\textup{)}  that the  condition $\sF_r=\{0\}$ implies that $\si_1$ is observable. But then,  as in the proof of Theorem \ref{innerdeltaH}, we  construct the system $\si_{10}$, show that $M_{\si_{10}}$ is an isometry, and  conclude that $\d_\si(H_\circ)=0$.
\end{rem}

\begin{cor}\label{coinner} Let   $\th\in \sS(\sU, \sY)$ be co-inner, and let   $\si$ be a minimal realization of  $\th$. Then $\textup{RI}_\si^\circ$ consists of a single element, $H_\bu$ say, and  $\d_{\si^*}(H_\bu^{-1})=0$.
\end{cor}
\begin{proof}
 Assume   $\th\in \sS(\sU, \sY)$ is co-inner. Then $I_\sY-\th(\z)\th(\z)^*=0$ almost everywhere for $\z\in \BT$, and hence the space  $\sF_l$ consists of the zero element only.  The latter implies (see Remark \ref{zerocoupl}) that  $\textup{RI}_\si^\circ$ consists of a single element.

Next we use Theorem \ref{thm:Ric2}. Recall that  $\th^\sim (\l)=\th(\bar{\l})^*$ for $\l \in \BD$. The fact that $\th$ is co-inner, implies that $\th^\sim$ is inner. Indeed, we have
\begin{align*}
\mbox{$\th$ is co-inner}\ &\Longleftrightarrow\    \th(\z) \th(\z)^*=I\  \mbox{almost everywhere on $\BT$}\\
&\Longleftrightarrow \ \th(\bar{\z}) \th(\bar{\z})^*=I\  \mbox{almost everywhere on $\BT$}\\
&\Longleftrightarrow \ \th^\sim(\z)^*  \th^\sim(\z)=I\  \mbox{almost everywhere on $\BT$}\\
&\Longleftrightarrow \ \mbox{$\th^\sim$ is inner}.
\end{align*}
Since  $\si$ is a minimal realization of  $\th$, the system
$\si^*$ is a minimal realization for  $\th^\sim$. Now let $H_\bu$
be the (unique) element in $\textup{RI}_\si^\circ$. From
\eqref{sisi*} it follows that $H_\bu^{-1}$ belongs to
$\textup{RI}_{\si^*}^\circ$. But $\si^*$ is a minimal realization
of an inner function. Hence, $\d_{\si^*}(H_\bu^{-1})=0$ by
Theorem~\ref{innerdeltaH}.
\end{proof}

Note that the first statement in the above corollary can also be proved by using the duality argument used in the second paragraph of the above proof.

In general, the second part of Theorem \ref{innerdeltaH} is not true for a co-inner function. See Example \ref{ex:co-inner} in the next section.

\begin{rem}\label{rem:new2} Finally, again with minor changes,  one can   prove that Corollary \ref{coinner} remains true if the condition $\th$ is co-inner is replaced by the condition that $\sF_l=\{0\}$.
\end{rem}

\setcounter{equation}{0}
\section{Examples}\label{sec:examples}
In this section we present a few examples. Throughout $\th$ is a Schur class function and  $\si=(A,B,C, D; \sX, \sU, \sY) $ is a minimal realization of $\th$.  In the first three examples the state space $\sX$ will be  finite dimensional. In that case a positive operator on $\sX$  will be bounded and boundedly invertible, and the Riccati equality can be rewritten as
\begin{equation}\label{RicE2}
\alpha_\si(H)-\beta_\si(H)^* \delta_\si(H)^{[-1]}\beta_\si(H)=0,
\end{equation}
where
\begin{align*}
&\alpha_\si(H) =  H - A^\ast H A - C^\ast C, \quad  \beta_\si(H) =
D^\ast C + B^\ast H A, \\ &\hspace{2cm}\delta_\si(H) =  I - D^\ast
D - B^\ast H B.
\end{align*}
Furthermore, if $\sX$ is   finite dimensional, then  $H\in \textup{RE}_\si$ if and only if  $H$ is a positive  operator on $\sX$, the operator   $\delta_\si(H)$ is nonnegative, and   $H$ satisfies \eqref{RicE2}. Similarly, if $\sX$ is   finite dimensional, then  $H\in \textup{RI}_\si$ if and only if  $H$ is a positive  operator on $\sX$,  the operator $\delta_\si(H)$ is nonnegative, and
\begin{equation}\label{RicI2}
\alpha_\si(H)- \beta_\si(H)^* \delta_\si(H)^{[-1]}\beta_\si(H)\geq 0.
\end{equation}
As before, the symbol ${[-1]}$ denotes the Moore--Penrose inverse.

\medskip
\begin{exa}\label{ex:*optimal} \textup{We present a simple scalar  example showing that the maximal solution
in $\textup{RI}_\si$ may not belong  to $\textup{RE}_\si$. To do
this  we use the  scalar function $\th$ given by
\cite[eq.~(3.3)]{AKP97}, i.e.,
\[
\th(\l)=(2\l +4)(\l +8)^{-1}.
\]
From \cite{AKP97} we know that $\th$ is a Schur class function (in fact, $|\theta(\lambda)| \leq 6/7 < 1$ for all $\lambda \in \BD$) and a minimal realization of $\th$ is given by
\begin{equation}\label{systemex1}
\si=(-\frac{1}{8}, 1, \frac{3}{16}, \frac{1}{2}; \BC, \BC,\BC).
\end{equation}
For this choice of $\si$  the set $\textup{RE}_\si$ is a singleton
and $\textup{RI}_\si$ is an interval
\begin{equation}\label{RicEI2}
\textup{RE}_\si=\left\{\frac{3}{64}\right\} \ands \textup{RI}_\si=\left[\frac{3}{64}, \frac{3}{4}\right].
\end{equation}
In particular, the maximal solution  $H_\bu$ of the Riccati inequality does not belong $\textup{RE}_\si$. }
\end{exa}

To prove \eqref{RicEI2} let $h$ be a positive real number viewed as a positive operator on $\BC$. Then
\[
\alpha_\si(h)  =  \frac{9}{64} \left(7 h - \frac{1}{4}\right), \quad  \beta_\si(h) = \frac{1}{8} \left(\frac{3}{4}  -  h\right),\quad
\delta_\si(h)= \frac{3}{4}  -  h.
\]
Note that $\delta_\si(h)\geq 0$ if and only if  $h\leq 3/4$. The Moore-Penrose inverse of $\delta(h)$ is given by
\[
\delta_\si(h)^{[-1]} =
\left\{
\begin{array}{cl}
 \left(\frac{3}{4}  -  h\right)^{-1}  \ \ &(h \neq 3/4)\\[.1cm]
0  \ \ &(h = 3/4)
\end{array}
\right. \ \ : \ \BC \rightarrow \BC.
\]
For $h=3/4$ the right hand sides of  both \eqref{RicE2} and \eqref{RicI2} are zero, and the   left hand sides are strictly positive.   Thus $3/4\not \in \textup{RE}_\si$ and $3/4\in \textup{RI}_\si$. Next, let $0<h<3/4$.  Then, respectively,   \eqref{RicE2} and \eqref{RicI2}  reduce to
\begin{align}
 \frac{9}{64} \left(7 h - \frac{1}{4}\right)&-\frac{1}{64}\left(\frac{3}{4}  -  h\right)^2\left(\frac{3}{4}  -  h\right)^{-1}=0, \label{RicE3}\\
  \frac{9}{64} \left(7 h - \frac{1}{4}\right)&- \frac{1}{64}\left(\frac{3}{4}  -  h\right)^2\left(\frac{3}{4}  -  h\right)^{-1}\geq 0. \label{RicI3}
\end{align}
Equation \eqref{RicE3} has $h=3 /64$ as its unique solution in the interval $0<h< 3/4$, which proves the first equality in  \eqref{RicEI2}. All solutions $h$ of    \eqref{RicI3} are given by $h\geq 3 /64$. Together with  $0<h\leq 3/4$ this yields  the second equality in  \eqref{RicEI2}.

Let $\si^*$ be the adjoint of the system $\si$ given by \eqref{systemex1}, i.e.,
\[
\si^*=(-\frac{1}{8},  \frac{3}{16}, 1, \frac{1}{2}; \BC, \BC,\BC).
\]
For this choice we have the following analogue of \eqref{RicEI2}
\begin{equation}\label{RicEI3}
\textup{RE}_{\si^*}=\left\{\frac{4}{3}\right\} \ands \textup{RI}_\si=\left[\frac{4}{3}, \frac{64}{3}\right].
\end{equation}

By Theorem \ref{thm:Ric2} the second identity in \eqref{RicEI3} follows from the second identity in \eqref{RicEI2}.  The first identity in \eqref{RicEI2} cannot be obtained  in this way but this identity is proved in a similar way as the first identity in  \eqref{RicEI3} is proved. We omit the further details. Note that in this case
\begin{equation}\label{RicE4}
\textup{RE}_{\si^*} \not = \{H^{-1}\mid H\in \textup{RE}_{\si}\}.
\end{equation}


\begin{exa}\label{ex:uniquesol}  \textup{We consider  the scalar function
\[
\th(\l)=\frac{\l ab}{1- \l^2 ab},\quad \mbox{where} \quad
0<a<b<1,\quad  a^2+b^2=1.
\]
The function $\th$ is a Schur clas function,  and $\th$ is the transfer function of the system $(A,B,C, D; \BC^2, \BC, \BC)$, where
\[
A=\begin{bmatrix} 0& a\\ b&0\end{bmatrix},\quad  B=\begin{bmatrix} 0 \\ a \end{bmatrix}, \quad
C=\begin{bmatrix} 0 &  b \end{bmatrix}, \quad D=0.
\]
The system $\si$ is a passive minimal realization of $\th$,  and hence $H_1=I_{\BC^2}$ is a solution of the  Riccati   equality associated to $\si$. We shall see that  there are three other solutions, namely:
\begin{align*}
&H_2 = \frac{1}{a^2} \begin{bmatrix}\left(1-ab\right) \frac{b}{a}&\left(b-a\right) \sqrt{ \frac{b}{a}}\\ \left(b-a\right) \sqrt{ \frac{b}{a}}& 1-ab \end{bmatrix}, \quad H_3= \frac{1}{a^2} \begin{bmatrix}\left(1-ab\right) \frac{b}{a} &-\left(b-a\right) \sqrt{ \frac{b}{a}}\\ -\left(b-a\right) \sqrt{ \frac{b}{a}}& 1-ab \end{bmatrix},\\
&\hspace{4cm} H_4=\frac{1}{a^4}\begin{bmatrix} b^4&0\\ 0 &a^2b^2 \end{bmatrix}.
\end{align*}
We shall see that   $H_1\leq H_j \leq H_4$, $j=2, 3$. It follows from Theorem  \ref{thm2b} that  $H_1=I_{\BC^2}$ is the minimal optimal solution, and thus  $H_1$ is the  minimal element  in $\textup{RI}_\si$. It turns out  that  $H_4$ also belongs to $\textup{RI}_\si$ and is the maximal element  in $\textup{RI}_\si$. }
\end{exa}


\noindent To derive the results mentioned above, put
\[
H = \begin{bmatrix} x_1 &  x_2 \\ x_3 &  x_4 \end{bmatrix}.
\]
By assumption   $H$ is positive definite. In particular, $\overline{x_3}= x_2$.
In this case we have
\begin{align*}
\alpha_\si(H) &=
\begin{bmatrix}
x_1-b^2x_4 & x_2-ab x_3\\
 x_3-ab x_2 & x_4-a^2 x_1 -b^2
\end{bmatrix},\quad \beta_\si(H)=\begin{bmatrix} ab x_4 & a^2 x_3\end{bmatrix},\\
&\hspace{2cm}\delta_\si(H) = 1-a^2 x_4.
\end{align*}
Recall that $\delta_\si(H) = 1-a^2 x_4$ is required to be non-negative, and the associate Riccati equality is   the identity
\begin{equation}\label{RIcEex3}
\begin{bmatrix}
x_1-b^2x_4 & x_2-ab x_3\\
 x_3-ab x_2 & x_4-a^2 x_1 -b^2
\end{bmatrix}- (1-a^2 x_4)^{[-1]}
\begin{bmatrix}a^2b^2 x_4^2   & a^3 b x_4 x_3\\
a^3 b x_2 x_4&a^4 x_2 x_3 \end{bmatrix}=0.
\end{equation}
Since $H$ is positive definite, $x_4>0$. Together  with $1-a^2 x_4 \geq 0$ this implies that $0<x_4\leq a^{-2}$.   But $x_4= a^{-2}$ is excluded, because in that case  the Riccati equation \eqref{RIcEex3} has no solution which can be proved by direct checking.  Therefore  we may assume that  $0<x_4< a^{-2}$, and hence the Moore-Penrose inverse in \eqref{RIcEex3} is a usual inverse. But then, with  elementary  computations or using the computer algebra program Mathematica, it is straightforward to show that the matrices $H_j$, $j=1,2,3,4$, are the only solutions of the Riccati equality \eqref{RIcEex3}.

Since $H_2$ and $H_3$ have the same diagonal entries, neither $H_2\leq H_3$ nor $H_3\leq H_2$. Indeed, if  $H_2\leq H_3$, then $H_3-H_2$ is a nonnegative operator of the form \eqref{A:matpos} with zero diagonal entries. But then, by Proposition \ref{prop:A1}, the off diagonal entrties are zero too, and hence $H_2=H_3$ which is not true. In a similar way one shows that $H_3\leq H_2$ is excluded. The fact that $H_4$ is the maximal element  in $\textup{RI}_\si$ can be obtained from Theorem  \ref{thm:Ric2} by showing that $H_4^{-1}$ is the minimal element of $\textup{RI}_{\si^*}$. Note that is this case
\[
\si^*=\left(\begin{bmatrix} 0& b\\ a&0\end{bmatrix}, \begin{bmatrix} 0 \\ b \end{bmatrix},
\begin{bmatrix} 0 & a \end{bmatrix}, 0; \BC^2, \BC, \BC\right).
\]
\medskip
\begin{exa}\label{ex:co-inner}
\textup{Let $\th(z)=\begin{bmatrix}z&0\end{bmatrix}$. Then  $\th(z)\th(z)^*=1$ for each  $z\in \BT$, and thus $\th$ is co-inner. We show that   the statement in the second part of Theorem \ref{innerdeltaH} does not hold for this  co-inner function $\th$.  To do this put
\begin{align*}
A=0: \BC\to \BC, \quad B&=\begin{bmatrix}1&0\end{bmatrix}: \BC^2\to \BC, \quad  C=1: \BC\to \BC, \\ D&=\begin{bmatrix}0&0\end{bmatrix}: \BC^2\to \BC.
\end{align*}
Then the system $\si=(A,B,C, D; \BC, \BC^2, \BC)$ is a minimal realization of $\th$ and its system matrix
\[
M(\si)=\begin{bmatrix} 0&1&0\\ 1&0&0  \end{bmatrix}\quad \mbox{is
a co-isometry}.
\]
Thus,   by Corollary  \ref{coinner} above, $\textup{RI}_\si^\circ=\{1\}$.  But in this case
\[
\d_\si(1)=I_{\BC^2}-D^*D-B^*B=\begin{bmatrix}1&0\\0&1\end{bmatrix}-\begin{bmatrix}1\\0\end{bmatrix}\begin{bmatrix}1&0\end{bmatrix}=\begin{bmatrix}0&0\\0&1\end{bmatrix}.
\]
Thus $\d_\si(1)$ is non-zero.}
\end{exa}

\begin{exa}\label{ex:lambdaK} \textup{We present an example of a minimal passive system  $\si$ such that  $H\in \textup{RI}_\si$  while  $\si_H$  is not minimal. In particular, $\textup{RI}_\si^\circ$ will be a proper subset of $ \textup{RI}_\si$.   The transfer function $\th$ of the system involved will be of the form $\th(\l)= \l K$, where $K$ is a contraction. The latter allows us to  use results from \cite[Section 2.3]{AKP05}}.
\end{exa}

\smallskip
Throughout $\ell_+^2$  is the Hilbert space  of  all complex valued sequences that are square summable in absolute value. Furthermore,  $R$ and $S$ are the  linear operators acting in $\ell_+^2$ defined by
\begin{align}
\sD(R)&=\{x\in \ell_+^2 \mid (x_0,2  x_1, 3x_2, \ldots)\in
\ell_+^2\}, \quad Rx=(0, x_0,2  x_1, 3x_2, \ldots);\label{defR1}\\
\sD(S)&= \{\l v+x\mid \l\in \BC, \ v=\left(1, 1/2,1/3,
\ldots\right), \quad  x\in  \sD(R)\},\ands \nn\\ & S(\l v+x)=\l
e_0 +Rx, \quad \mbox{where}\quad e_0=(1, 0,0, 0,
\ldots).\label{defS}
\end{align}
The operators $R$ and $S$ are both closed  densely defined linear operators and both are one-to-one.  Furthermore,
\[
\sD(R)\subset \sD(S),\quad S|_{\sD(R)}=R, \quad  \sD(R)\not = \sD(S)\not = \ell_+^2,   \quad \im S=\overline{\im S}= \ell_+^2.
\]
Since $S$ is densely defined, its adjoint $S^*$ is well defined. In what follows $\sU$ and $\sY$ denote  the spaces $\sD(R)$ and  $\sD(S^*)$ endowed with the corresponding graph norms. Thus
\begin{align*}
\|x\|_\sU&=\left(\|x\|^2+\|Rx\|^2\right)^{1/2}, \quad x\in \sD(R),\\
\|x\|_\sY&=\left(\|x\|^2+\|S^*x\|^2\right)^{1/2}, \quad x\in \sD(S^*).
\end{align*}
Next, put $\sX=\ell_+^2$, and define the canonical embeddings
\begin{align*}
&\t_\sU:\sU\to  \begin{bmatrix} \sX \\ \sX \end{bmatrix},\quad \t_\sU u= \begin{bmatrix} u\\ R u \end{bmatrix}, \quad u\in  \sD(R),\\
&\t_\sY:\sY \to  \begin{bmatrix} \sX \\ \sX \end{bmatrix},\quad \t_\sY y= \begin{bmatrix} y\\  S^*y \end{bmatrix}, \quad y\in   \sD(S^*).
\end{align*}
Note that both $\t_\sU$ and  $\t_\sY$ are isometries. We also need
the projections
\[
\Pi_1= \begin{bmatrix} I & 0 \end{bmatrix}: \begin{bmatrix} \sX \\ \sX \end{bmatrix}\to \sX \ands  \Pi_2= \begin{bmatrix} 0 & I \end{bmatrix}: \begin{bmatrix} \sX \\ \sX \end{bmatrix}\to \sX.
\]
Given we these operators we consider  the system
\begin{align}
&\si=(0, B, C, 0; \sX, \sU, \sY), \quad \mbox{where}\nn\\
&\hspace{2cm}B=\Pi_1\t_\sU:\sU\to \sX \ands
C=\left(\Pi_2\t_\sY\right)^*:\sX\to \sY.\label{defexsi}
\end{align}
Clearly, $B$ and $C$ are contractions, and hence the system matrix $M(\si)$
 is a contraction too. It follows that  $\si$ is passive. Note that $\im B=\sD(R)$, and hence $\overline{\im B}=\overline{\sD(R)}=\sX$.  Furthermore, $\im C^*=\im S$, and thus $\overline{\im C^*}=\sX$. The latter implies that  $C$ is one-to-one. We conclude that the system $\si$ is minimal. Finally,  the transfer function  of $\si$ is the Schur class function $\th$ given by $\th(\l)=\l CB$.

 Next we consider a second system
\begin{align}
&\wh{\si}=(0, \wh{B}, \wh{C}, 0; \sX, \sU, \sY), \ \mbox{where}\nn\\
&\hspace{1.5cm}\wh{B}=\Pi_2\t_\sU: \sU\to \sX \ands \wh{C}=(\Pi_1\t_\sY)^*: \sX\to \sY. \label{defexwhsi}
\end{align}
Note that $\im \wh{B}=\im R$. Thus $\im R$ is not dense in $\sX$,
and hence the system $\wh{\si}$ is not minimal.

 \begin{prop}\label{prop:BC1} The systems $\si$ and $\wh{\si}$  defined by \eqref{defexsi} and \eqref{defexwhsi}, respectively, have the same transfer function, and the operator  $S$ defined by \eqref{defS} is a pseudo-similarity from $\si$ to $\wh{\si}$.
 \end{prop}

 Now put $H=(S^*S)^{1/2}$. Then we know from \cite[Proposition 4.5]{AKP06} that $H\in \textup{KYP}_\si$, and thus  $H\in \textup{RI}_\si$, by Theorem \ref{thm:RIKYP}. Moreover,  $\si_H$ is unitarily equivalent to $\wh{\si}$. In particular, $\si_H$ is not minimal, and thus $H\not \in \textup{RI}_\si^\circ$.

\medskip
The above proposition  can be obtained by applying the result of \cite[Section 2.3.1]{AKP05}. For sake of  completeness we present the proof. In order to do this it will be convenient  first to prove the following  lemma.

\begin{lem}\label{lem:tauY*1} Let $a\in \sD(S)$ and let $x\in \sX$. Then the following three statements are  equivalent:
\begin{align}
&\textup{(a)}\ \t_\sY^*\Pi_2^* a=x, \quad \textup{(b)}\
\t_\sY^*\Pi_1^* Sa=x , \label{equivax1}\\ &\textup{(c)}\ x\in
\sD(SS^*)\quad \text{and}\quad   (I+SS^*)x=Sa. \label{equivax2}
\end{align}
In particular, $\t_\sY^*\Pi_2^* a=\t_\sY^*\Pi_1^* Sa$ for each   $a\in \sD(S)$.
\end{lem}
\begin{proof}
We split the proof into two parts.

\smallskip
\noindent\textsc{Part 1.} We  prove the equivalence of  items (a) and (c). To do this we use the fact  (see formula (5.9) in \cite[page 168]{Kato66}) that there exist (unique) vectors $x_1\in \sD(S^*)$ and $x_2\in \sD(S)$ such that
\begin{equation}\label{x12a}
\Pi_2^* a=\begin{bmatrix}0\\ a  \end{bmatrix}= \begin{bmatrix}x_1\\ S^*x_1  \end{bmatrix}+  \begin{bmatrix}-Sx_2\\ x_2  \end{bmatrix}.
\end{equation}
Note that $\t_\sY^*\Pi_2^* a=x_1$. The identity  \eqref{x12a} is equivalent to
\begin{equation}\label{x12b}
x_1=Sx_2 \ands a= S^*x_1+x_2.
\end{equation}
Since $a\in \sD(S)$ and  $x_2\in \sD(S)$, the second identity in \eqref{x12b} shows that $S^*x_1=a-x_2\in  \sD(S)$. Thus $x_1\in  \sD(SS^*) $ and using the first identity in \eqref{x12b} we obtain
\[
Sa=SS^*x_1 +Sx_2= (I+SS^*)x_1.
\]
We conclude that  $(I+SS^*)x=Sa$ with $x=x_1$.

Conversely, assume $x \in \sD(SS^*) $ satisfies  $ (I+SS^*)x=Sa$. Put $x_1=x$ and define $x_2=a-S^*x_1$. Then $x_2\in \sD (S)$ and
\[
Sx_2=Sa- SS^*x_1=Sa -\left(I+SS^*\right)x_1 +x_1=x_1.
\]
Thus the two identities in  \eqref{x12b} are satisfied which implies that \eqref{x12a} holds. Hence $\t_\sY^*\Pi_2^*a=x_1=x$. \epr

\smallskip
\noindent\textsc{Part 2.}  We prove    $\t_\sY^*\Pi_1^* Sa=\t_\sY^*\Pi_2^*  a$.
Again using formula (5.9) in \cite[page 168]{Kato66}, there exist (unique) vectors $x_1\in \sD(S^*)$ and $x_2\in \sD(S)$ such that
\begin{equation}\label{x12c}
\Pi_1^*  Sa=\begin{bmatrix} Sa\\0\end{bmatrix}=
\begin{bmatrix}x_1\\ S^*x_1  \end{bmatrix}+  \begin{bmatrix}-Sx_2\\ x_2  \end{bmatrix}.
\end{equation}
The latter identity is equivalent to
\begin{equation}
\label{x12d}
Sa=x_1- Sx_2 \ands x_2=-S^*x_1.
\end{equation}
Since $x_2\in \sD(S)$, the second identity in \eqref{x12d} shows that  $x_1\in \sD(SS^*)$ and $ Sx_2=-SS^*x_1$. Using this fact  the first identity in \eqref{x12d} yields
\[
Sa=x_1+SS^*x_1=(I+SS^*)x_1.
\]
But then we can apply the result of the previous part to show that $\t_\sY^*\Pi_2^*a=x_1$.   On the other hand, from \eqref{x12c} it follows that  $\t_\sY^*\Pi_1^*Sa$ is also equal to $x_1$. Hence we have $\t_\sY^*\Pi_1^*Sa= \t_\sY^*\Pi_2^*a$ as desired.  Together the two parts prove the lemma.
\end{proof}

\medskip
\noindent\textit{Proof of Proposition \ref{prop:BC1}}.  Recall
that $S$ is one-to-one and has a dense range. Therefore, since
$\si$ and $\wh{\si}$ are given by \eqref{defexsi} and
\eqref{defexwhsi}, respectively, it suffices to show that
\begin{equation}\label{pseudosim2}
B\sU\subset \sD(S), \quad  \wh{B}=SB\ands  \wh{C}Sa=Ca \quad \left(a\in \sD(S)\right).
\end{equation}
Take $u\in \sU\, (=\sD(R))$. Then
\begin{align*}
&Bu=\Pi_1 \t_\sU u= \Pi_1\begin{bmatrix}u\\ Ru \end{bmatrix}=u\in  \sD(R)\subset \sD(S) \ands\\
&SBu=Su=Ru=\Pi_2\begin{bmatrix}u\\ Ru \end{bmatrix}=\Pi_2\t_\sU u=\wh{B}u.
\end{align*}
This proves the first part of \eqref{pseudosim2}. To prove the
second part,  let $a\in \sD(S)$. Using Lemma~\ref{lem:tauY*1} we
have
\[
\wh{C}Sa=(\Pi_1\t_\sY)^*Sa=\t_\sY^*\Pi_1^*Sa=\t_\sY^*\Pi_2^*a=(\Pi_2\t_\sY)^*a=Ca.
\]
Hence $S$ is a pseudo-similarity from $\si$ to $\wh{\si}$. In particular, the two systems have the same transfer function, i.e., $CB=\wh{C}\wh{B}$.\epr

\appendix
\section{}\label{App}
\renewcommand{\theequation}{A.\arabic{equation}}
\setcounter{equation}{0}
In this appendix we review a number of results regarding $2\ts 2$ nonnegative operator matrices that are used in the present paper.  In particular, we shall consider  Schur complements for such operators. Throughout  we assume that $\a: \sX\to \sX$, $\b: \sU\to \sX$, $\d: \sU\to \sU$ are  bounded  Hilbert space operators and $T$ is the  bounded  operator defined by
\begin{equation}\label{A:matpos}
T=\begin{bmatrix}
 \a     &  \b  \\
\b^*      &\d
\end{bmatrix}: \begin{bmatrix} \sX\\ \sU \end{bmatrix} \to
\begin{bmatrix} \sX\\ \sU \end{bmatrix}.
\end{equation}

\begin{prop}\label{prop:A1} The operator $T$ is nonnegative if and only if  $\a$ and $\d$ are nonnegative and there exists a contraction $\ga: \sX \to \sU$  such that
\begin{itemize}
\item[\textup{(a)}] $\kr \ga \supset \kr \a$ and $\im \ga \subset \overline{\im \d}$,
\item[\textup{(b)}] $\b^*=\d^{1/2}\ga \a^{1/2}$.
\end{itemize}
Moreover, in that case  $\ga $ is uniquely determined by conditions \textup{(a)} and  \textup{(b)}.
\end{prop}
If  $T$ is nonnegative and $\ga$ is the contraction satisfying  the  two conditions in  the above proposition, then we call $\ga$ the \emph{minimal contraction} determined  by~$T$.
For the proof of the  proposition see  the proof  of \cite[Theorem XVI.1.1]{FF90}, of  \cite[Lemma 2.4.4]{BW11} or   of    \cite[Lemma A.1]{DriRov10}.

Assume $T$ is nonnegative, and let $\ga$ be the minimal contraction determined  by $T$. Then the operator $\de$ on $\sX$ given by
\begin{equation} \label{defSchurcpl}
\de=\a^{1/2}(I-\ga ^*\ga)\a^{1/2}
\end{equation}
is called  the \emph{Schur complement} of $T$ supported by $\sX$. If $\d$ is invertible, then $\de=\a-\b\d^{-1}\b^*$, which is the classical Schur complement formula (see, e.g., \cite[Lemma A.1.2]{FFGK98}). From formula \eqref{defSchurcpl} it follows that the  Schur complement $\de=0$ if and only if $\ga$ is a partial isometry with initial space equal to $\overline{\im \a}$.

\begin{prop}\label{prop:A2} Let $T$ be  nonnegative. The Schur complement  of $T$ supported by $\sX$ is also given by
 \begin{equation}
\label{defaltDe}
\lg \de x, x\rg=\inf  \left\{ \lg T \begin{bmatrix} x\\u \end{bmatrix}, \begin{bmatrix} x\\u \end{bmatrix}\rg    \mid  {u\in \sU}\right\}, \quad  x\in \sX.
\end{equation}
\end{prop}
\begin{proof}
By direct checking one proves that
\[
T=\begin{bmatrix}  I_\sX& \a^{1/2}\ga^*\\ 0 & \d^{1/2}\end{bmatrix}
\begin{bmatrix}\de&0\\ 0 & I_\sU\end{bmatrix}
\begin{bmatrix}  I_\sX& 0\\ \ga\a^{1/2} & \d^{1/2}\end{bmatrix}.
\]
Using this identity  we see that
\begin{align*}
 \lg T \begin{bmatrix} x\\u \end{bmatrix}, \begin{bmatrix} x\\u \end{bmatrix}\rg  &= \lg \begin{bmatrix}\de&0\\ 0 & I_\sU\end{bmatrix}\begin{bmatrix}  I_\sX& 0\\ \ga\a^{1/2} & \d^{1/2}\end{bmatrix} \begin{bmatrix} x\\u \end{bmatrix},
\begin{bmatrix}  I_\sX& 0\\ \ga\a^{1/2} & \d^{1/2}\end{bmatrix} \begin{bmatrix} x\\u \end{bmatrix}\rg\\
 &=  \lg \begin{bmatrix}\de&0\\ 0 & I_\sU\end{bmatrix}
 \begin{bmatrix} x\\ \ga\a^{1/2}x+\d^{1/2}u \end{bmatrix},
\begin{bmatrix} x\\ \ga\a^{1/2}x+\d^{1/2}u \end{bmatrix} \rg\\
 &=\lg \de x, x\rg +\|\ga\a^{1/2}x+\d^{1/2}u\|^2.
\end{align*}
Thus for $x\in \sX$ and $u\in \sU$ we have
\begin{equation}\label{ineqxu}
\lg \de x, x\rg\leq  \lg T \begin{bmatrix} x\\u \end{bmatrix}, \begin{bmatrix} x\\u \end{bmatrix}\rg \leq \lg \de x, x\rg +\|\ga\a^{1/2}x+\d^{1/2}u\|^2.
\end{equation}
Now fix $x\in \sX$. Recall that  $\im \ga\subset \overline{\im \d}=\overline{\im \d^{1/2}}$. Thus $\ga\a^{1/2}x\in \overline{\im \d^{1/2}}$. It follows that there exist a sequence $u_1, u_2, \dots$ in $\sU$ such that
\[
\lim_{n\to \iy} \|\ga\a^{1/2}x+\d^{1/2}u_n\| =0.
\]
But then \eqref {ineqxu} shows that \eqref{defaltDe} holds.
\end{proof}

The notion of a Schur complement is closely related that  of  a
\emph{shorted operator} as defined by M.~G.~Kre\v{\i}n  in
\cite{MKrein47}. In fact, if $T$ is nonnegative, then  $\de$ is
the Schur complement  of $T$ supported by $\sX$ if and only if
\[
\begin{bmatrix} \de&0\\ 0&0 \end{bmatrix}
\]
is the shorted operator corresponding to  $T$ and $\sX$. This follows from formula \eqref{defaltDe}; see Section 2 in \cite{Arl08} for further details.

{\it{Acknowledgments}}.  The authors thank the referee for his/her
careful reading of the paper. The remarks  of the referee are
incorporated in the paragraph  directly after
Lemma~\ref{lem:minopt} and  in Remarks \ref{rem:new1} and
\ref{rem:new2}.

\end{document}